\documentclass{amsart}

\usepackage[utf8]{inputenc}

\usepackage{todonotes}

\usepackage{amsfonts,amssymb,amsthm,amsmath}
\usepackage{graphicx,multicol,color}

\usepackage{mathrsfs}

\definecolor{bleu_sombre}{rgb}{0,0,0.6}
\definecolor{rouge_sombre}{rgb}{0.8,0,0}
\definecolor{vert_sombre}{rgb}{0,0.6,0}
\usepackage[plainpages=false,colorlinks,linkcolor=bleu_sombre,citecolor=rouge_sombre,urlcolor=vert_sombre,breaklinks]{hyperref}

\theoremstyle{plain}
\newtheorem{theorem}{Theorem}
\newtheorem*{theorem*}{Theorem}
\newtheorem{corollary}[theorem]{Corollary}
\newtheorem{lemma}{Lemma}[section]
\newtheorem{proposition}{Proposition}
\theoremstyle{remark}
\newtheorem{remark}{Remark}
\theoremstyle{definition}

\theoremstyle{definition} \newtheorem*{key}{Key words}
\theoremstyle{definition} \newtheorem*{ams}{A.M.S. Classification}

\newcommand{\be} {\begin{equation}}
\newcommand{\ee} {\end{equation}}
\newcommand{\bea} {\begin{eqnarray}}
\newcommand{\eea} {\end{eqnarray}}
\newcommand{\Bea} {\begin{eqnarray*}}
\newcommand{\Eea} {\end{eqnarray*}}

\usepackage{bbm}
\newcommand{\ind}{\mathbbm{1}} 

\newcommand{\ag}{\alpha}

\newcommand{\eg}{\varepsilon}

\newcommand{\FF}{\mathcal{F}}
\newcommand{\GG}{\mathcal{G}}

\newcommand{\TT}{\mathcal{T}}
\newcommand{\VV}{\mathcal{V}}

\newcommand{\ZZ}{\mathcal{Z}}

\newcommand{\bT}{\boldsymbol T}
\newcommand{\bZ}{\boldsymbol Z}
\newcommand{\bW}{\boldsymbol W}

\newcommand{\bbZ}{\boldsymbol{\overline{Z}}}

\newcommand{\N}{\mathbb{N}}
\newcommand{\Q}{\mathbb{Q}}
\newcommand{\R}{\mathbb{R}}
\newcommand{\Z}{\mathbb{Z}}

\renewcommand{\P}{{\mathbf P}}
\newcommand{\E}{{\mathbf E}}

\RequirePackage{stmaryrd}
\newcommand{\entiers}[2]{\llbracket#1,#2\rrbracket}

\renewcommand{\tilde}{\widetilde}
\def\w{\widetilde}

\newcommand{\st}{~\text{s.t.}~}

\begin{document}
\title{Queueing for an infinite bus line and aging branching process}
\author{Vincent Bansaye \& Alain Camanes}
\date{\today}
\address{}
\email{}
\bibliographystyle{abbrv}
\renewcommand{\labelenumi}{$\roman{enumi}.$}
\renewcommand{\labelitemi}{$\bullet$}
\renewcommand{\labelitemii}{$\ast$}

\maketitle

\begin{abstract}
We study a queueing system with Poisson arrivals on a bus line indexed by integers. The buses move at constant speed 
to the right and the time of service per customer getting on the bus is fixed.
The customers arriving at station $i$ wait for a bus if this latter is less than  $d_i$ stations before, where $d_i$ is non-decreasing. 
We determine the asymptotic behavior of a single bus and when two buses eventually coalesce almost surely by coupling arguments. Three regimes appear, two of which leading to a.s. coalescing of the buses.\\
The approach relies on a connection with aged structured branching processes with immigration and varying environment. We need to prove a Kesten Stigum type theorem, i.e. the a.s. convergence of the successive size of the branching process normalized by its mean. The technics developed combines a spine approach for multitype branching process in varying environment and geometric ergodicity along the spine to control the increments of the normalized process. 
\end{abstract}

\begin{key} Queuing systems, Polling, Multitype Branching processes, Immigration, Inhomogeneity, Coupling.
\end{key}
\begin{ams}60J80, 60K25, 60F05, 60F10.
\end{ams}

\tableofcontents

\section{Introduction \& Main results}

The main motivation of this paper is the study of queueing systems where customers arrive on the half line of integers (a bus line) and the traveling servers (the buses) move at a constant speed to the right. The time service for each customer (the time required to get on the bus) is assumed to be constant for simplicity. In the model considered here, a customer arriving in station $i\in \N$ is waiting at this station if the distance from this customer to the next bus is less than $d_i$. We describe now  this process, its  link with aging branching processes and the main results  of the paper.

\subsection{Queuing system for a bus line}\label{sec:introqueueing}
We consider an infinite bus line, where stations are labeled by $\N=\{0, 1,2,3,\ldots\}$. The buses all start from station $0$, the depot, and no customer  get on at this station. 
To each station $i \geq 1$, we associate an independent Poisson Point Processes $\bT^{(i)} = (T^{(i)}_k : k\geq 1)$ on $[0,\infty)$ with intensity $\alpha \in (0,1)$:
$T^{(i)}_k$ is the time arrival of  the $k$-th customer at station $i$.\\
The discipline $(d_i : i \geq 1)$  of the queuing is   a  sequence of positive integers such that
\begin{equation}\tag{A}\label{as:di}
d_1 = 1,  \qquad d_2=2,  \qquad d_{i} \leq d_{i+1}\leq d_{i}+1 \quad (i \geq 2).
\end{equation}
When a customer arrives at station $i$, if the next bus is $d_i$, or more, stations far from him, then he decides not to take the bus and he  leaves the queue. Thus, the customers who are indeed waiting at the station $i$ are those
arrived at this station after a bus has left the station $i-d_i$ and before it has left station $i$. Notice that, when $d_i=1$, customers wait at station $i$ only if the bus has already left the previous station and we have assumed $d_i>1$ for $i>1$ to simplify the links with aging branching processes and avoid the degenerate case when only switch over times matter. \\
The discipline $d_i=i$ is a first relevant example. Then, the  customers wait  the first bus at the station $i$ as soon as the first bus has left the depot  (time $0$) and has not already left station $i$.     This polling system for one server with infinitely many stations is a sequential  clearing process
and  has been studied in \cite{K1} when Poisson arrivals happen on the half line $[0,\infty)$ with rate $\lambda$. Then, the position $P_t$ of the bus (server) at time $t$ is equivalent to $\log (t)/\lambda$ when $t\rightarrow \infty$. When the bus can move to the left, analogous transient behavior occur and we refer to \cite{K2, F} for the study of the greedy walker respectively on $\Z$ and $\R$.\\
The model when $d_i = d \wedge i$ is an other interesting discipline. Then, the customers wait at  station $i$  if the distance to the next bus is strictly less than $d$ stations, soon as $i\geq d$.  Many works have been dedicated to models with a finite number of queues and servers moving from one queue to another. In particular, a vast literature deal with customers arriving on a circle while the servers visit the queues cyclically. We refer to \cite{Ta, Borst, Boxma, ReviewPolling} and reference therein for an overview about polling models, see also the discussion below. The stability of the queueing then depends on the number of queues and time service. In this case, ballistic and logarithm  regimes appear, which   are linked to the subcritical and supercritical regimes of the associated branching structure. This happens also in our model. \\
The class of discipline (\ref{as:di}) that we consider allows to describe intermediate regimes between the two previous examples.
 Then, the farther the station is from the depot, the more customers may wait. Since sequence $(d_i)$ is non-decreasing from Assumption~\eqref{as:di}, the case when $d_i$ is bounded amounts to consider, for $i$ large enough, $d_i$ constant and the main difficulties arise when $d_i$ goes to infinity. Several extensions of discipline (\ref{as:di}) could be considered and we make a few more precise comments along the paper.
Another motivation for considering the discipline (\ref{as:di})  is the  study of the motion of a bus behind another bus,  when the distance between them increases like the sequence $(d_i)_i$, see the last part (Section~\ref{sec:buses_two}) of the paper. \\
The position of the  bus $k\in \{1, 2 \ldots \}$ at time $t$ is denoted by $P_t^{(k)}$  and the time when it leaves the station $i$ by $H_i^{(k)}$. Between two stations, the buses move at constant speed $\tau$
and  for any $i\geq 0$:
\Bea
P_{t}^{(k)}&=&i + (t - H_i^{(k)}) / \tau \quad \text{for } \ t \in [H_i^{(k)},H_i^{(k)}+ \tau];\\
P_t^{(k)}&=&i+1 \quad \text{for } \ t \in [H_i^{(k)}+\tau, H_{i+1}^{(k)}].
\Eea
For simplicity, we assume that the time service  for each customer (to get on a bus)  is constant and equal to one, so  the time spent in a  station  is  the number of customers who have got on. \\
In this paper, we consider only two buses and  $k \in \{1, 2\}$. We assume  that bus $1$ starts from the depot at time $0$ and bus  $2$ starts from the depot at time $\mu > 0$: 
$P_0^{(1)} = 0,$   $P_t^{(2)} = 0 \text{ for any  } t  \in [0,\mu]$.
 The successive positions of the buses can be  defined recursively. To do so, we introduce the label $F_i^{(k)}$ of the first potential  customer waiting for bus $k$ at the station $i$: this customer arrived after the departure of the bus $k$ from station $i - d_i$ and after the passage of the previous bus  (if there is one, i.e. for the second bus). We then consider the set of customers 
 $\{F_i^{(k)}+j : j\in \mathcal I_i^{(k)}\}$
 arriving after this first customer (including himself) and before the departure of the bus from the station. This set may be empty if the first potential customer arrived too late. It provides the label of the last (potential) customer $M_i^{(k)}$  getting on the bus and  the number of customers getting on the bus, which yields the time spent by the bus in station $i$. \\
 More precisely, the times $(H_i^{(1)} : i \geq 0)$ are defined recursively by
\begin{equation}\label{eq:ind}
\begin{split}
H_0^{(1)} &= 0\ a.s.\\
F_i^{(1)} &= \min\left( j \geq 1 : T_j^{(i)} >H_{i-d_i}^{(1)}\right) \\
\mathcal I_i^{(1)} &= \{k \in \N : \forall j \leq k,\, T_{F_i^{(1)}+j}^{(i)} \leq H_{i-1}^{(1)} + \tau + j\} \\
M_i^{(1)} &= F_i^{(1)} - 1 + \#\ \mathcal I_i^{(1)} \\
H_i^{(1)} &= H_{i-1}^{(1)} + \tau + (M_i^{(1)} - F_i^{(1)} + 1),
\end{split}
\end{equation}
where $\#\ \mathcal I_i^{(1)}$ denotes the number of elements of $I_i^{(1)}$. First note that $0\leq i-d_i<i$ under $(\ref{as:di})$. Morerover, since $\alpha > 0$, then $(T_j^{(i)})_j$ goes to 
infinity a.s. and $F_i^{(1)}$ is a.s. finite. Finally, since $\alpha < 1$, we will explain in Section~\ref{sec:bus_aging} and Appendix~\ref{sec:queues}, using a standard link with  Galton-Watson process, that $\mathcal I_i^{(1)}$ is  bounded and $M_i^{(1)}$ is also finite a.s.
\begin{figure}[ht!]
\includegraphics[width=0.7\textwidth]{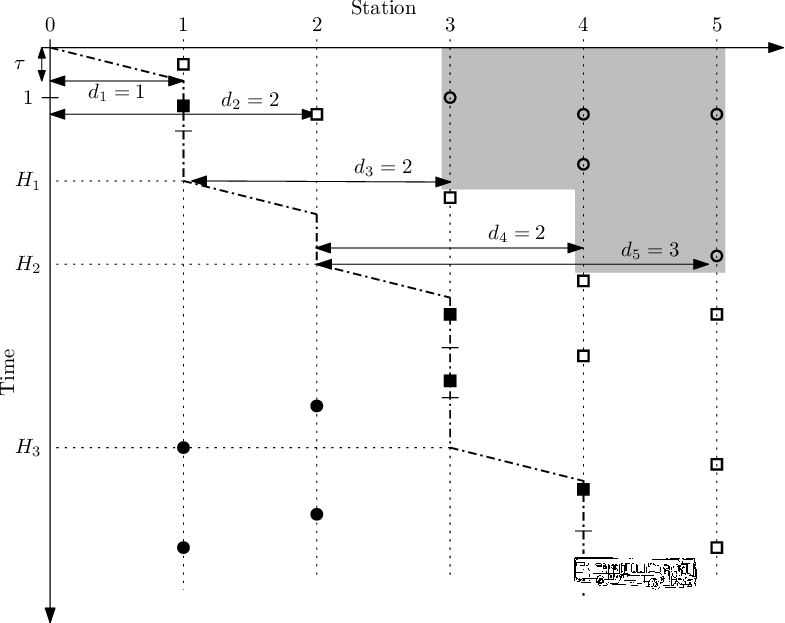}
\caption{\small Successive positions of a single bus following discipline $(d_i)$. Customers who are not getting on the bus, either arrive when the bus was too far ($\circ$), or when the bus 
has already left ($\bullet$). Customers getting on the bus arrived at the station either before the bus arrival ($\square$), or when another customer of this station was getting on ($\blacksquare$).}
\end{figure}

Let us define similarly the motion of the second bus but now customers may have been taken by the first bus. Using  the label  $M_i^{(1)}$ of the last customer getting on  the bus $1$, we set for $i\geq 1$,
\begin{equation}\label{eq:ind2}
\begin{split}
H_0^{(2)} &= \mu\ a.s.\\
F_i^{(2)} &= \min\left( j > M_i^{(1)} : T_j^{(i)}> H_{i-d_i}^{(2)}\right ) \\
\mathcal I_i^{(2)} &= \{k \in \N : \forall j \leq k,\, T_{F_i^{(2)}+j}^{(i)} \leq H_{i-1}^{(2)} + \tau + j\} \\
M_i^{(2)} &= F_i^{(2)} - 1 + \#\ \mathcal I_i^{(2)} \\
H_i^{(2)} &= H_{i-1}^{(2)} + \tau + (M_i^{(2)}-F_i^{(2)}+1).
\end{split}
\end{equation}

\subsection{Main results}
The position of the bus $1$ (a bus alone on the line) is described using  the following theorem, whose proof is deferred to Section~\ref{sec:buses_one}. When $(d_i)_i$ is bounded (and non decreasing from Assumption~\eqref{as:di}), it is stationary and for simplicity we focus in that case on $d_i= d \wedge i$. For any integer $a\geq 0$, let $\varrho(a) \in (0,\infty]$ be the  the largest root of $$P_a = X^{a+1} - m \sum_{k = 0}^a X^k, \quad \text{where }  m=\frac{\alpha}{1-\alpha}.$$ 
Thus, $\varrho(a)$   is  the largest eingenvalue of the Leslie matrix of size $a+1$, see   Appendix~\ref{sec:mean_asymptotics}.  In the next statement, $d\geq 2$ is a constant.

\begin{theorem*}
\noindent $(i)$ If $d_i = d \wedge i$  and $\varrho(d-2)<1$, then there exists $c>\tau$ such that
\[
H_i^{(1)} \sim_{i\rightarrow \infty} ci \quad a.s.
\]
Moreover, $(H_i^{(1)}-ic)/\sqrt{i}$ converges weakly to a Gaussian r.v. with positive variance.

\noindent $(ii)$ If $d_i = d \wedge i$ with   $\varrho(d-2)>1$ or $d_i\rightarrow \infty$, then
\[
H_i^{(1)} \sim_{i\rightarrow \infty} m_i W \quad a.s., 
\]
where $W$ is a finite positive random variable  and  $m_i$ satisfies the following recursion
\begin{equation}\label{eq:mn}
m_0 = 1; \qquad   m_{n} = m \sum_{i=0}^{d_{n+1}-2} m_{n-1-i} \quad (n \geq 1).
\end{equation}
\end{theorem*}
This dichotomy $(i-ii)$ between the ballistic and logarithm speed  is inherited from the subcritical and supercritical behavior of the underlying branching structure and is classical for polling models, see in particular \cite{MeyResing,  Ta} and discussion below. Here the main difficulty arises when  $d_i$ goes to infinity and in particular when it does not go to infinity fast enough to behave as the greedy walker.  The critical case and the heavy traffic limit are not investigated in this paper, see in particular \cite{VdM} for cycling polling models. \\
In the case $(ii)$, if  $d_i= d\wedge i$ then $m_i$ grows geometrically with rate $\varrho(d-2)\in (1,  m + 1)$; while if  $d_i\rightarrow \infty$, then $m_{i+1}/m_i$ goes to $m + 1$. 
We refer to  Lemma~\ref{lem:mean_asymptotic} for finer  results in that case. In particular,  we give an analytical criterion on $d_i$ to ensure that   $m_i \sim c(m+1)^i$ for some $c>0$, which corresponds
to the asymptotic behavior of the clearing sequential process
and  greedy walk \cite{K1,K2}($d_i=i$). \\

The station where the two buses coalesce is defined by 
\[
N = \min( i \geq 0 : H_i^{(1)}\geq H_i^{(2)} )=\min(i \geq 1 ;  \exists t>0 : P_t^{(1)}=P_t^{(2)}=i\},
\]
which is infinite when the buses never coalesce. 
Obviously, $\P(N<\infty)>0$ and we say that the two buses coalesce a.s. iff $\P(N<\infty)=1$.
In Section~\ref{sec:buses_two}, we prove the following result using coupling arguments and the discipline of the buses:
\begin{theorem*}
$(i)$ If $d_i = d \wedge i$ and $\varrho(d-2)<1$, then the two buses coalesce a.s. \\
$(ii)$ If $d_i = d \wedge i$ and $\varrho(d-2) > 1$, then the two buses do not coalesce a.s. \\
$(iii)$ If $d_i\rightarrow \infty$ as $i \rightarrow \infty$, then the two buses coalesce a.s.
\end{theorem*}

\subsection{Aging branching processes}\label{sec:introaging}

The behaviors of the buses rely on the study of the branching structure associated to the motion of a single bus. Indeed, during the time a customer gets on the bus, new customers 
may arrive both at the same station and at the following stations\ldots  
Customers waiting at  station $i$  take  a bus and then provoke some queuing (only) in a set of stations following their position, whose size depends on $i$ through the discipline $(d_k)_k$. These 
customers queuing in these successive stations $i+k$ can be seen as offsprings
and the integer $k$ as an age for the original customer. That leads us to study an  aging branching structure where  
individuals (customers) give birth during a life period whose length depends on their birth time (their station of arrival).  We describe in this section this branching process, which may be interesting for itself. The  link with the bus line is specified in Proposition~\ref{bus_gw}, Section \ref{sec:bus_aging}.

\medskip

We consider a branching process where each individual reproduces independently with reproduction law given by an integer valued random variable $R$ and may die depending on his age. We write  $a_n$ the maximal age one  individual may have in generation $n$ and we assume that
\begin{equation}\tag{B}\label{as:age}
\forall n \in \N,\quad  0\leq a_n \leq a_{n+1} \leq a_n + 1.
\end{equation}
This provides an aging structure to the population.
This is also the counterpart of Assumption~\eqref{as:di} above for the discipline of the queueing system. More specifically, both sequences are linked
by $a_n = d_{n+1} - 2$ for $n\geq 2$ (see Section~\ref{sec:bus_aging} for details) as soon as the right hand side is non-negative. \\
For each $n,\, a \in \N,$ we denote by $\bZ_n(a)$ the number of individuals in generation $n$ whose age is equal to $a$ and $\bZ_n$ the total size of the population in generation $n$. The aging branching process with immigration $(\bZ_n : n \in \mathbb N)$ is initialized with a bounded random vector $(\bZ_0(a): a\in \N)$ where $Z_0(a) = 0$ if $a > a_0$ and  defined recursively as follows:
\begin{subequations}
\label{eq:reproduction}
\begin{align}
\bZ_n(a) &= \bZ_{n-1}(a-1) \ind_{ a\leq a_{n}} \quad \text{for } a\geq 1; \\
\bZ_n(0) &= \sum_{j = 1}^{\bZ_{n-1}} R_{j,n} + I_n;\\
\bZ_n &:= \sum_{a=0}^{a_n} \bZ_n(a); 
\end{align}
\end{subequations}
where $(R_{j,n},\, j,n \in \N)$ are i.i.d.  integer valued random variables distributed as $R$ and $(I_n,\, n \in \N)$ are independent  random variables independent of $(R_{j,n},\, j,n \in \N)$.

In Section~\ref{sec:immigration}, we  prove a Kesten Stigum theorem and specify the distribution of traits among the population. We  write $m=\E(R)$  and assume that
\begin{equation}\tag{C}\label{as:immigration}
\E(R^2) < \infty,  \qquad
\limsup_{n \rightarrow \infty} \frac{\log(\E(I_n))}{n}<1+m, \qquad \ \P(I_n \geq x)\geq \P(I\geq x),
\end{equation}
for any $x\geq 0$, where $I$ is a non-negative random variable and $\P(I>0)>0$.  These assumptions allow to use $L^2$ computations 
and guarantees that the process survives but the immigration term is not exploding fast enough to change  its growth. 
\begin{theorem*} Assume that $a_n \rightarrow \infty$.\\
$(i)$ The process $(\bZ_{n}/m_n : n \in \N)$ converges a.s. to a finite positive random variable.\\
$(ii)$ For each  $a\geq 0$, the process $(\bZ_n(a) / \bZ_n : n \in \N)$ converges a.s. to $m/(m+1)^{a+1}$.
\end{theorem*}
The proof is first achieved without immigration and the quantity $m_n$ is the mean number of individuals living at generation $n$ in that case, see Section~\ref{sec:mean_estimates}. Of course, $m_n$ is determined by 
 Equation~\eqref{eq:mn} replacing $d_{n+1} - 2$ by $a_n$. Let us stress that when $a_n$ goes to infinity, the asymptotical distribution of age in the population  coincide with the immortal case $a_n=\infty$ but the growth  of the population $(m_n)_n$ may differ, see  Lemma~\ref{lem:mean_asymptotic}  in Appendix for details.

\subsection{Connection between queuing and branching process and discussion}
The connection between branching processes and queueing systems is known from a long time and can be found for example in \cite{Feller} (chapter XIV). Basically the offspring of an individual is the number of customers arriving during the time when this individual is served.\\
Multitype branching processes have  played a key role in the analysis of polling models and immigration allows to capture polling models with non-zero switch-over times, see e.g. \cite{MeyResing}.
 Using multitype branching process with final product allows to  describe the busy period time and could provide an efficient way to extend this work to random time service, see   \cite{Dyak}  with zero switch-over time and \cite{Vat} with non-zero switch-over time, both in random environment. 
The stable case corresponds to subcritical branching process, and the heavy traffic limit (when the load tends to $1$) has also been investigated using near critical
branching processes \cite{VdM}. This would provide another stimulating regime for the  model considered here.

A vast literature is dedicated to multitype branching processes (see e.g. \cite{Mode, AN}) with a finite number of types and their applications to population dynamics (see e.g. \cite{HJV,Kimmel}) and queueing systems (see e.g \cite{LectureVatutin}). In particular, they give a convenient way to model population without interactions and take into account a bounded age structure influencing the reproduction law of the individuals. The associated mean matrix is then called Leslie matrix and Perron Frobenius theory enables to describe the asymptotic behavior of the process. In particular, the convergence of the renormalized size of a branching process can be found in \cite{AN,KLPP}.\\
Much less is known about infinite number of types and varying environment since the spectral approach and martingale techniques cannot be extended easily. The pioneering works of e.g. Moy \cite{Moy} and Kesten \cite{Kesten} provide some theoretical extensions of limit theorems and extinction criteria for a denumerable number of types, while multitype branching processes with a finite number of types have been well studied for stationary ergodic environment (see e.g. \cite{AK1, AK2}). \\
 In varying environment and finite dimension, Coale Lopez  theorem   \cite{Seneta} about product of matrices can be used to describe the mean behavior of the process and prove asymptotic results on the branching process via weak ergodicity, see \cite{Jones} and  \cite{Cohn}. Here we have to deal both with varying environment and infinite dimension. Necessary and sufficient conditions for a.s. and $L^p$ convergence of martingales associated to time and space harmonic vectors have been obtained in \cite{BCN} in the case of countable set of types. But they rely on the ergodic behavior of the first moment semigroup and their application in  our framework  seems delicate. Nevertheless, it could help for example to relax the moment assumptions made here. \\
We develop here a slightly different approach and control the $L^2$ moments of the successive increments of the renormalized process $\bZ_n/\E(\bZ_n)$ without immigration. This technique is inspired by the spine approach initiated by \cite{CR, LPP, KLPP} (and the so-called many-to-one formula and formula for forks). 
We provide an expression of the transition matrix of the Markov process along the spine following \cite{Ban}. Thanks to a Doeblin's condition, we get the geometric ergodicity of this  Markov chain which enables us to control finely the first moment semigroups and then the increments of the renormalized process. We believe this method is rather efficient and could be generalizable in several ways. \\
Let us finally mention that taking into account time non-homogeneity  and immigration is more generally motivated by self excited process (like Hawkes process) where the immigration is the excitation and the branching structure is the cascade effect provoked by this excitation. Here letting $a_n$ goes to infinity makes the range of memory of the process be very large.

\subsection{Organization of the paper}

The first part (Section~\ref{sec:aging}) of this work is dedicated to the study of the aging structured branching population $\bZ_n$. Branching process without immigration is first 
studied and  Theorem~\ref{thm:main} states the $L^2$ convergence of the normalized process and the asymptotic distribution of the ages among the population. Then, to derive the behavior in
the presence of immigration (see Theorem~\ref{thm:immigration}), we sum the descendants of each immigrant. We  conclude this section with the behavior of the cumulative number of newborns (see Corollary~\ref{thm:newborn}), which is the key quantity used for the study of the motion of the buses. All along the paper, bold letters will be dedicated to processes with immigration, e.g. $Z_n$ is for the size of the population in generation $n$ without immigration and $\bZ_n$ is for the one with immigration. \\
In Section~\ref{sec:bus_aging}, we  specify the link between bus line and the aging branching process. \\
In Section~\ref{sec:queueing}, we first study the progression of one single bus  (Theorem~\ref{thm:1bus}). We then apply these results to determine when two successive buses coalesce a.s. (Theorem ~\ref{thm:coalescence}) by coupling arguments.

\section{Aging branching process with maximal age going to infinity }\label{sec:aging}
 In this section, we assume $\eqref{as:age}$ and  focus  on maximal ages going to infinity:
\[
\lim_{n\rightarrow \infty} a_{n} =\infty.
\]
We first note that the means of $Z_n$  and  $\bZ_n$  go to infinity and the process is supercritical.
Indeed, for any $a_0 \geq 0$ and  $n$ large enough, $Z_n(.)$ is stochastically larger than a multitype branching process whose mean matrix
is the squared Leslie's matrix of size $a_0 + 1$; see Appendix~\ref{sec:mean_asymptotics} and below for details and  finer results.

\subsection{Aging branching process without immigration}\label{sec:sans_immigration}
First define the aging branching process without immigration. The process $(Z_n(.) : n \in\mathbb N)$ is defined by $Z_0(0) = 1$, $Z_0(a)=0$ for $a\geq 1$ and for $n \geq 1$,
\begin{subequations} \label{eq:reproduction_sans_im}
\begin{align}
Z_n(a) &= Z_{n-1}(a-1) \ind_{a \leq a_{n}} \quad \text{for } a\geq 1;\\
Z_n(0) &= \sum_{j = 1}^{Z_{n-1}} R_{j,n}; \label{eq:nv_nes_sans_im} \\ 
Z_n &:= \sum_{a=0}^{a_n} Z_n(a); 
\end{align}
\end{subequations}
where $(R_{j,n},\, j,n \in \N)$ are i.i.d. random variables distributed as $R$. The non-extinction event is defined by
\[
\text{Nonext} := \{\forall n \geq 0 : Z_n > 0\}.
\]

\smallskip

We give now the asymptotic behavior of $Z_n$ under the $L^2$ condition
$$\sigma^2 = \E[(R-m)^2]<\infty.$$
It is used in next Section~\ref{sec:immigration} to derive the counterpart  when the immigration is non zero.

\begin{theorem}\label{thm:main}
\noindent $(i)$ The event $\emph{Nonext}$ has positive probability.

\noindent $(ii)$ The process $(Z_n/\E(Z_n) : n \in \N)$ converges a.s. and in $L^2$ to a non-negative r.v. $W$, which is positive a.s. on \emph{Nonext}.

\noindent $(iii)$ For each $a\in \N$, the process $(Z_n(a) / Z_n : n \in \N)$ converges a.s. to $m / (m+1)^{a+1}$ on \emph{Nonext}.
\end{theorem}

We prove first  $(i)$. The two following subsections~\ref{sec:main_ii} and \ref{sec:main_iii} are respectively devoted to the proofs of $(ii)$ and $(iii)$.

\begin{proof}[Proof of Theorem~\ref{thm:main} $(i)$]
 Since $(a_n)$ goes to infinity, for $n_0$ large enough $\varrho(a_{n_0}) > 1$, where $\varrho(a)$ has been defined in the Introduction
 and is the greatest eigenvalue of the Leslie Matrix sized $a+1$. Moreover, with positive probability, $Z_{n_0}(0) \neq 0$. Then, $(Z_n)_{n\geq n_0}$ is 
 lower bounded by a multitype branching process starting from one  newborn, with fixed maximal age $a_{n_0}$ and reproduction law $R$. 
Recalling that the mean matrix of this latter is a primitive matrix with maximal eigenvalue greater than $1$  and that the second moment of $R$ is finite ensures that
with positive  probability, it  grows to infinity, see the supercritical regime  in \cite{AN},  Theorem 1, Chapter V - Section 6. This proves $(i)$.
\end{proof}

We prove in Section~\ref{sec:mean_estimates} that $m_{n+1}/m_n\rightarrow 1$, so Theorem~\ref{thm:main} ensures that
the ratio of consecutive population sizes $Z_{n+1}/Z_n$ goes to $m+1$ on Nonext. However, the process may grow slower than $(m+1)^n$. In fact, two asymptotic regimes appear, depending whether $(a_n)$ goes fast enough to infinity or not.

\begin{corollary}\label{cor:regimes}
The process $M_n = (m+1)^{-n} Z_n$
is a supermartingale which converges a.s. to a finite non-negative random variable $M$. Furthermore,
\begin{align*}
(i)& \text{ if } \liminf_{n\rightarrow \infty} a_n/ \log n >1/ \log(m+1) \text{, then } \{M > 0\} = \emph{Nonext}. \\
(ii)& \text{ if } \limsup_{n\rightarrow \infty} a_n / \log n < 1/ \log(m+1) \text{, then } M= 0 \text{ a.s.} \qquad \qquad \qquad \qquad \qquad \qquad 
\end{align*}
\end{corollary}
\noindent Notice that, when $a_i = i$, each individual is immortal  and we obtain a Galton-Watson process with mean $m+1$, which grows geometrically with rate $m+1$ as $n\rightarrow \infty$. The result above shows at which speed $(a_n)$ has to go infinity so that such a growth still holds. The proof of this statement is postponed to Section~\ref{sec:cor_regimes}.

\subsubsection{Genealogical description}

To achieve the proof of Theorem~\ref{thm:main}, we use a genealogical description of the aging branching process. 
Following notations of Ulam-Harris-Neveu, we introduce 
$$\mathcal I = \{\varnothing \} \cup \bigcup_{n\geq 1} \{0,1,2,\ldots\}^n$$
endowed with the natural order $\preccurlyeq$: $u\preccurlyeq v$ if there exists $w\in \mathcal I$ such that
$v=uw$. Let $|u|=n$ denotes the length of $u$ when $u=u_1\ldots u_n\in \mathcal I$ and $|\emptyset| = 0$. \\
We define now $(\mathcal T, \ZZ)$ the genealogical tree labeled by the age of the individuals 
where $\TT$ is a random subset of $\mathcal I$ and $\ZZ=(\ZZ(u) : u\in \TT) $ is a collection of integers giving the ages. This can be achieved
recursively as follows: 
\begin{itemize}
\item $\varnothing \in \TT$ and $\ZZ(\varnothing) = 0$.
\item $u0 \in \TT$ iff $u\in \TT$ and $\ZZ(u)+1\leq a_{\vert u\vert+1}$. Then, $\ZZ(u0)=\ZZ(u)+1$.
\item For any $j\geq 1$ and $u\in \TT$, $uj\in \TT \quad \text{iff} \ j \leq R_{ Cous_{\ell}(u), \vert u \vert}$ where
\[
Cous_{\ell}(u)= \sharp \{ v \in \TT : v=v_1\ldots v_{\vert u\vert }, \exists j \leq \vert u\vert, \forall k \leq j,\, v_k \leq u_k\}.
\]
Then $\ZZ(uj)=0$. 
\end{itemize}

\begin{figure}[ht!]
\includegraphics[width=0.9\textwidth]{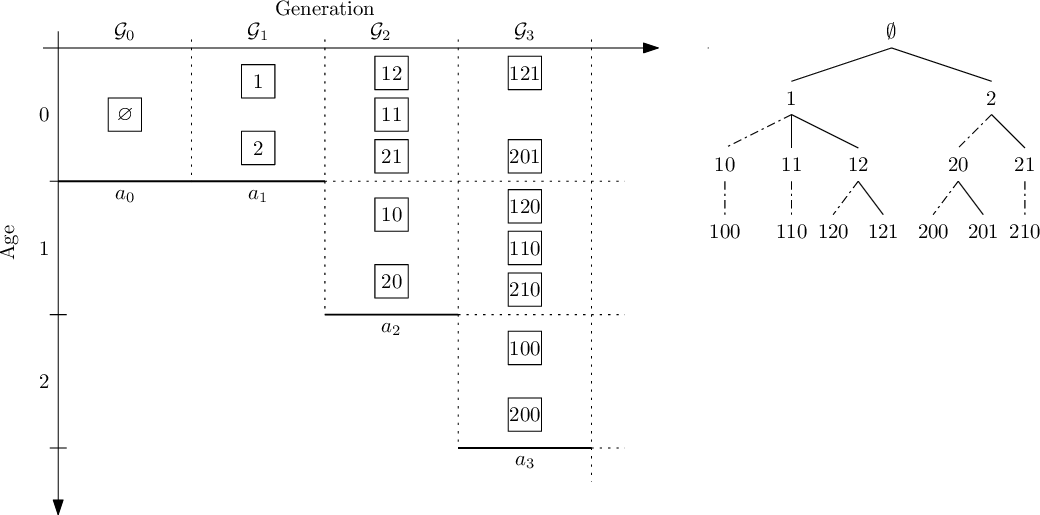}
\caption{Example of a population evolution with labeled individuals, starting with a single individual and with maximal age $a_0 = a_1 = 0,\, a_2 = 1$ and $a_3 = 2$. In the tree description, dash-dotted lines are used to denote surviving individuals.}
\end{figure}
\noindent Thus, the offspring of an individual $u$
 are labeled by $ui$ with $1\leq i \leq R_{ Cous_{\ell}(u), \vert u \vert}$, while $u$ becomes
$u0$ in the next generation if it survives.
The set of individuals in generation $n$ is denoted $\GG_n$ and the set of the offspring of $u$ in the next generation (including himself if it survives) is denoted $X(u)$. Finally, in each generation $n$, the part of the population which survives in the next generation is denoted by $\GG_n^+$, while the part of the population which dies is denoted $\GG_n^-$. More precisely,
\begin{align*}
\GG_n &:= \{u \in \TT \st |u| = n\}, \\
X(u)&:=\{ v \in \TT : \vert v\vert =\vert u\vert +1, u \preccurlyeq v\}, \\
\GG_n^+ &:= \{u \in \GG_n \st \ZZ(u) + 1 \leq a_{n+1}\}, \\
\GG_n^- &:= \{u \in \GG_n \st \ZZ(u) + 1 > a_{n+1}\}.
\end{align*}
Then, the population  $(Z_{n}(a) : a\geq 0)$ in generation $n$ is described by the measure:
$$\ZZ_n = \sum_{u\in \GG_n} \delta_{\ZZ(u)}.$$
Indeed, $Z_n(a)=\ZZ_n(\{a\})=\#\{ u \in \GG_n : \ZZ(u)=a\}$ is the size 
of the population with age $a$ in generation $n$ and
 $Z_n=\ZZ_n(\N)=\# \GG_n$. Finally, we write $\mathcal F_n$ the $\sigma$ algebra generated by $((u,X(u)) :  u \in \TT, |u| \leq n)$.

\subsubsection{Moment estimates}\label{sec:mean_estimates}

Before  the proofs, we start with some simple estimates on $\E[Z_n]$. To evaluate the mean behavior of the process, we consider the mean number of individuals at generation $n$ which are descendants of a single individual aged $a$ at generation $i$. Thus, let us note for every $0 \leq i \leq n$ and $0 \leq a \leq a_i$,
\begin{align}\label{def:mina}
m_{i,n}(a) &= \E[Z_n | \ZZ_i = \delta_a],
\end{align}
For sake of simplicity, when the initial individual is a newborn, we  write
\begin{align}\label{def:min}
m_{i,n} = m_{i,n}(0), \qquad m_n=m_{0,n}=m_{0,n}(0) = \E[Z_n  | \ZZ_0 = \delta_0].
\end{align}
Finally, we consider the mean number of individuals aged $b$ in generation $n$, which are descendants from a single individual with age $a$ in generation $i$. For $0 \leq i \leq n$, $0 \leq a \leq a_i$, $0 \leq b \leq a_n$, let
$$m_{i,n}(a,b)= \E[Z_n(b) | \ZZ_i = \delta_a].$$
 
We provide here some simple but useful estimates on these mean quantities. The proofs rely on  the semi group property 
of the first moment operators $(m_{i,n}(.,.))_{i\leq n}$, which is inherited from the  Markov and branching properties of the process $(Z_n)_n$.
\begin{lemma}\label{lem:mean_branching}
\noindent $(i)$ For every $n > 0$, $m_n(0,0) = m \ m_{n-1}$.

\noindent $(ii)$ For all $0 \leq i < n$, $i < k < n$ and $a \leq a_i$,
\be
\label{eq:mean_rec}
m_{i,n}(a) = \sum_{0 \leq \ell \leq a_k} m_{i,k}(a, \ell) m_{k,n}(\ell). 
\ee
\end{lemma}

\begin{proof}
$(i)$ Each individual in generation $n-1$ reproduces independently with the same mean number of offspring $m$ and gives birth to the newborns of generation $n$.

\noindent $(ii)$ Conditioning by the distribution of traits in generation $k$, we get the equality using the branching and Markov property of $(Z_n(.) : n \in \N)$.
\end{proof}

\smallskip

Let us describe the behavior of the mean  growth of the number of individuals.
\begin{lemma}\label{lem:mean_behavior}
\noindent $(i)$ For every $0 \leq k \leq n$, $m_{n-k}(0,0)\, m_{n-k,n} \leq m_n$.

\noindent $(ii)$ $m_{n+1}/m_n \ \rightarrow \ m + 1$ as $n \rightarrow \infty$.\\
More generally, for any $ k\geq 0$, $m_{n-k} / m_n \  \rightarrow \ (m+1)^{-k}$ as $n \rightarrow \infty$.

\noindent $(iii)$ For all $0\leq i\leq n-1$ and $0 \leq a\leq a_i$,
$$m\ m_{i+1,n}\leq m_{i,n}(a) \leq m_{i,n} \leq (m+1)m_{i+1,n}.$$

\noindent $(iv)$ For each $\varrho \in (1, m + 1)$, there exists a positive constant $\ag$ such that for $n$ large enough, $m_n \geq \ag \varrho^n$.
\end{lemma}

\begin{proof}
$(i)$ The first inequality is obtained from \eqref{eq:mean_rec} by replacing $k$ by $n-k$ and considering $i=0$, $a=0$
and 
focusing on
newborn of generation $n-k$ ($\ell=0)$. 

\smallskip

$(ii)$ We first observe that from generation $n$ to generation $n+1$, only individuals aged $a_n$ may die. Furthermore, each individual gives birth, on average, to $m$ children. Thus,
\[
(m+1) m_n - m_{n}(0,a_n)\leq m_{n+1} \leq (m+1) m_n.
\]
This yields the upper bound. For the lower bound, note that individuals aged $a_n$ in generation $n$ were newborn at generation $n-a_n$, so $m_n(0,a_n)=m_{n-a_n}(0,0)$. Thus, 
using $(i)$ with $k = a_n$, 
\[
\frac{m_n(0,a_n)}{m_n}=\frac{m_{n-a_n}(0,0)}{m_n} \leq \frac{1}{m_{n-a_n, n}}.
\]
Moreover, $m_{n-a_n, n} = (m+1)^{a_n}$ since starting from an individual aged $0$ in generation $n-a_n$, nobody goes beyond the maximal age (and thus die) in the time interval $[n-a_n,n]$. So, each individual is replaced by $m+1$ individuals in average. Adding that $a_n$ goes to infinity, $m_n(0,a_n) / m_n$ goes to $0$ and this yields the lower bound for $(ii)$. By iteration, we obtain the second part of $(ii)$.

\smallskip

$(iii)$ The bound $m_{i,n}(a) \leq m_{i,n}$ comes simply from the fact that an individual with age $0$ lives longer that an individual with age $a>0$. Then \eqref{eq:mean_rec} yields the two last inequalities recalling that $m_{i,i+1}(a,0) = m$ and $m_{i,i+1}(a)\leq m+1$.

\smallskip

$(iv)$ Recall that $\varrho(a)$ is the greatest eigenvalue of Leslie Matrix sized $a+1$. Since $(a_n)$ goes to infinity, $(\varrho(a_n))_n$ goes to $m + 1$, see Lemma~\ref{lem:technical}. Then, there exists $n_0$ such that $\varrho(a_{n_0}) > \varrho$. Considering again the coupling with a multitype branching process with $a_{n_0}+1$ types starting at generation $n_0$ with  individuals aged $0$ yields
\[
m_n \geq m_{n_0}(0,0) M_{n_0, n},
\]
where $M_{n_0,n}$ is the mean number of individuals of this multitype branching process. Using Perron-Frobenius Theorem, there exists a positive constant $c$ such that $M_{n_0, n} \sim c \varrho(a_{n_0})^n$, which gives the expected result.
\end{proof}

\subsubsection{Proof of Theorem~\ref{thm:main} $(ii)$}\label{sec:main_ii}

The proof relies on the following lemma, which focuses on a typical individual, following the inspiration of spine decomposition (see e.g. \cite{CR, LPP, KLPP}) of branching processes.
For every $0\leq i  <n$ and $0 \leq a \leq a_i$, we define
\[
P_{i,n}(a,b)=m_{i,i+1}(a,b)\frac{m_{i+1,n}(b)}{m_{i,n}(a)},
\]
Then by recursion we can set $Q_{n,n}(a,a)=1$ and $Q_{n,n}(a,b)=0$ if $a\ne b$ and
\[
Q_{i,n}(a,b) := \sum_{0 \leq k \leq a_{i+1}} P_{i,n}(a,k) Q_{i+1,n}(k,b). 
\]
Identity \eqref{eq:mean_rec} ensures that $P_{i,n}$ is a Markov kernel. Then, $Q_{i,n}(.,.)$ is the semigroup of the associated  inhomogeneous Markov Chain between generations $i$ and $n$. We are proving in next Lemma~\ref{lem:ergodic} that it  satisfies Doeblin's condition and we are deriving geometric estimates for asymptotic behavior of the means $m_{i,n}$ via 
a many-to-one formula. Let us note that this technic is inspired by probabilistic methods, while the result  is analytical. It  can be extended for more general product of (non-negative) matrices in finite or infinite dimension using the  ergodicity  of  the Markov chain $Q$.

\begin{lemma}\label{lem:ergodic}
$(i$-\emph{Many-to-one formula}$)$. For all $0\leq i \leq n$ and $0 \leq a \leq a_i$,
\[
m_{i,n}(a,.)=m_{i,n}(a)Q_{i,n}(a,.).
\]
\noindent $(ii$-\emph{Doeblin's condition}$)$. For all $0\leq i < n$ and $0 \leq a\leq a_i$,
$$
P_{i,n}(a, .) \geq c \delta_0(.),
$$
where $c = m/(m+1)$ and $\delta_0$ is the Dirac mass at point $0$.

\noindent $(iii$-\emph{Geometrical Ergodicity}$)$. For any $i \leq n$ and probability measures $\mu$ and $\nu$ on $\{0,\ldots,a_i\}$,
\[
d_{TV}(Q_{i,n}(\mu, \cdot) , Q_{i,n}(\nu, \cdot)) \leq (1 - c)^{n-i} d_{TV}(\mu,\nu),
\]
where $d_{TV}(\cdot, \cdot)$ is the total variation distance on $\mathcal P(\N)$ defined by $$d_{TV}(\mu, \nu)=\frac{1}{2}\sum_{k \in \N} \vert \mu_k - \nu_k \vert.$$

\noindent$(iv)$ Setting
\[
E_{i,n}(a) :=\frac{ m_{i,n+1}(a)}{m_{n+1}}- \frac{m_{i,n}(a)}{m_{n}},
\]
there exists $C_m>0$,
such that for all $0 \leq i < n$, $0 \leq a \leq a_i$,
\[
\left\vert E_{i,n}(a)\right\vert \leq C_m \frac{ (1-c)^{n-i}}{m_i}.
\]
\end{lemma}

\begin{proof} We prove $(i)$ by induction and we assume that $m_{i+1,n}(k,b)=m_{i+1,n}(k) Q_{i+1,n}(k,b)$.
Using \eqref{eq:mean_rec}, we get
\Bea
\frac{m_{i,n}(a, b)}{m_{i,n}(a)}&=& \sum_{k \geq 0} \frac{m_{i,i+1}(a,k) m_{i+1,n}(k)}{m_{i,n}(a)} Q_{i+1,n}(k,b)\\
&=&\sum_{k \geq 0} P_{i,n}(a,k) Q_{i+1,n}(k,b),
\Eea
which proves $(i)$, since the right hand side is equal to $Q_{i,n}(a,b)$ by definition.

\noindent $(ii)$ From~\eqref{eq:mean_rec} applied to $k=i+1$,
\begin{align*}
m_{i,n}(a) &\leq m_{i,i+1}(a,0) m_{i+1,n}(0) + m_{i,i+1}(a,a+1) m_{i+1,n}(a+1) \\
&\leq m m_{i+1,n}(0) + m_{i+1,n}(a+1) \\
&\leq (m + 1) m_{i+1,n}(0),
\end{align*}
since an individual of age $0$ lives longer than an individual of age $a+1$. Thus, thanks to Lemma~\ref{lem:mean_behavior} $(iii)$,
\begin{align*}
P_{i,n}(a,0) &= m_{i,i+1}(a,0) \frac{m_{i+1,n}(0)}{m_{i,n}(a)} = m \frac{m_{i+1,n}}{m_{i,n}(a)} 
 \geq \frac{m}{m+1}.
\end{align*}

\noindent $(iii)$ Doeblin's condition obtained in $(ii)$ yields a contraction for the Total Variation distance (see e.g.~Lemma~4.3.13~in~\cite{CMR})
$$d_{TV}(P_{i,n}(\mu,.) ,P_{i,n}(\nu,.)) \leq (1-c)d_{TV}(\mu,\nu),$$
for any $0\leq i\leq n$ and by a classical induction
$$
d_{TV}(Q_{i,n}(\mu,.) ,Q_{i,n}(\nu,.)) \leq (1 -c)^{n-i}d_{TV}(\mu, \nu).
$$

\noindent $(iv)$ Using again \eqref{eq:mean_rec}, we get
\begin{align*}
m_{i,n+1}(a) &= \sum_{0\leq k \leq a_n} m_{i,n}(a, k) m_{n,n+1}(k) \\
&= m_{i,n}(a) \sum_{0 \leq k \leq a_n} Q_{i,n}(a, k) m_{n,n+1}(k).
\end{align*}
In particular,
\[
m_{n+1} = m_{0,n+1}(0) = m_n \sum_{0\leq k \leq a_n} Q_{0,n}(0,k) m_{n,n+1}(k).
\]
Thus,
\Bea
E_{i,n}(a)&=& \frac{m_{i,n}(a)}{m_n}\left[ \frac{\sum_{0\leq k\leq a_n} Q_{i,n}(a,k) m_{n,n+1}(k)}{\sum_{0\leq k\leq a_n} Q_{0,n}(0,k) m_{n,n+1}(k)}-1\right] \\
 &=& \frac{m_{i,n}(a)}{m_n} \cdot \frac{\sum_{0\leq k\leq a_n} (Q_{i,n}(a,k) - Q_{0,n}(0,k)) m_{n,n+1}(k)}{\sum_{0\leq k\leq a_n} Q_{0,n}(0,k) m_{n,n+1}(k)}. 
\Eea
Recalling that $m_{n,n+1}(k)$ is equal to $m$ or $m+1$ whether individual aged $k$ belongs to $\GG_n^-$ or $\GG_n^+$, we get
\[
|E_{i,n}(a)| \leq \frac{(m+1) m_{i,n}(a)}{m m_n} d_{TV}(Q_{i,n}(a, \cdot), Q_{0,n}(0, \cdot)).
\]
Now, from the triangular inequality,
\[
d_{TV} (Q_{i,n}(a, \cdot) , Q_{0,n}(0, \cdot))\leq d_{TV} (Q_{i,n}(a, \cdot) ,Q_{i,n}(0, \cdot)) + \sum_{k=1}^i d_{TV} (Q_{k,n}(0,\cdot) , Q_{k-1,n}(0,\cdot)).
\]
Then, using previous points $(ii)$ and $(iii)$,
\begin{align*}
d_{TV}(Q_{k,n}(0,\cdot) , Q_{k-1,n}(0,\cdot)) &= d_{TV} (Q_{k,n}(0,\cdot) , Q_{k,n}(P_{k-1,n}(0,\cdot),\cdot)) \\
&\leq (1 - c)^{n-k} d_{TV} (\delta_0 , P_{k-1,n}(0,\cdot)) \\
&\leq (1-c)^{n-k},
\end{align*}
and from Lemma~\ref{lem:mean_behavior} $(i)$ and $(iii)$, there exists a constant $C_m' > 0$ \st
\[
|E_{i,n}(a)| \leq C_m' (1 - c)^{n-i} \frac{m_{i,n}(a)}{m_n}\leq C_m'(1 - c)^{n-i} \frac{m_{i,n}}{m_n}\leq C_m'(1 - c)^{n-i} \frac{1}{m_{i}(0,0)}.
\]
Recalling that $m_{i}(0,0) = m_{i-1} m \geq m_i m / (m+1)$, we obtain $(iv)$.
\end{proof}

We prove now the convergence of $Z_n/m_n$ using the $L^2$ moments of its successive increments. This method has already been powerfully 
used for branching processes and also allows to obtain estimates on the speed of convergence, see e.g. \cite{Asm}.
\begin{proof}[Proof of Theorem~\ref{thm:main} $(ii)$]
To show that $(Z_n / m_n)$ converges a.s., we prove that
 \be
\label{cvser}
\sum_{n\geq 0} \E\left[ \left(\frac{Z_{n+1}}{m_{n+1}} -\frac{Z_n}{m_n}\right)^2\right]^{1/2}<\infty
 \ee
 by splitting this increment into the martingale increment and the drift term.
 First, we recall that
 $$Z_{n+1}= \sum_{u \in \GG_n} \sharp X(u)\quad \text{ and }\quad \GG_n=\GG_n^+\sqcup \GG_n^-.$$
  For $u \in \GG_n$, either $u$ belongs to $\GG_n^+$ and $\E( \sharp X(u) \ \vert \ \mathcal F_n)=m_{n,n+1}(\ZZ(u)) = m+1$, or $u \in \GG_n^-$ and $\E( \sharp X(u) \ \vert \ \mathcal F_n)=m_{n,n+1}(\ZZ(u)) = m$. Then
\be
\label{ajj}
\E(Z_{n+1} \ \vert \ \mathcal F_n)=\sum_{u \in \GG_n} m_{n,n+1}(\ZZ(u))=(m + 1) \sharp \GG_n^+ + m \sharp \GG_n^-\ee
and
\[
m_{n+1}=(m+1)\E(\#\GG_n^+)+m\E(\#\GG_n^-).
\]

Using Minkowski's inequality and denoting $\|\cdot\|_2$ the $L^2$ norm on r.v., we obtain
\begin{align}
\left\| \frac{Z_{n+1}}{m_{n+1}} - \frac{Z_n}{m_n} \right\|_2 &= \left\| \frac{\sum_{u\in \GG_n^+\sqcup \GG_n^-} \sharp X(u)}{m_{n+1}} -\frac{\E(Z_{n+1} \ \vert \ \mathcal F_n)}{m_{n+1}}+\frac{\E(Z_{n+1} \ \vert \ \mathcal F_n)}{m_{n+1}}- \frac{Z_n}{m_{n}} \right\|_2 \nonumber \\
&\leq \left\| \frac{\sum_{u\in \GG_n^{+}} [\sharp X(u)-(m+1)]}{m_{n+1}}\right\|_2 + \left\| \frac{\sum_{u\in \GG_n^{-}} [\sharp X(u)-m]}{m_{n+1}}\right\|_2 \nonumber \\
& \qquad + \left\| \frac{\E(Z_{n+1} \ \vert \ \mathcal F_n)}{m_{n+1}}- \frac{Z_n}{m_n}\right\|_2. \label{eq:ineg}
\end{align}

First remark that, when $u$ belongs to $\GG_n^+$, then $\sharp X(u)$ is distributed as $R + 1$. Thus, $(\sharp X(u) - (m + 1))_{u\in\GG_n^+}$ are independent, centered r.v. and
\[
\E\left[ \left(\sum_{u\in\GG_n^+} [\sharp X(u) - (m +1)]\right)^2\right] = \E[(R-m)^2] \E[\sharp \GG_n^+].
\]
Adding that the individuals in $\GG_n^+$ survive, $\E[\sharp \GG_n^+] \leq m_{n+1}$ and
\begin{equation}\label{eq:terme_1}
\left\| \frac{\sum_{u\in \GG_n^{+}} [\sharp X(u)-(m+1)]}{m_{n+1}}\right\|_2 \leq \frac{\sigma}{\sqrt{m_{n+1}}}.
\end{equation}
In the same manner, recalling that $m_{n+1}\geq m_nm$ from Lemma \ref{lem:mean_behavior} $(iii)$,
\begin{equation}\label{eq:terme_2}
\left\| \frac{\sum_{u\in \GG_n^{-}} [\sharp X(u)-m]}{m_{n+1}}\right\|_2 \leq \frac{\sqrt{m_n \E[(R-m)^2]}}{m_{n+1}} \leq \frac{\sigma}{\sqrt{m m_{n+1}}}.
\end{equation}
To study the last term, we write
\[
\frac{\E(Z_{n+1} \ \vert \ \mathcal F_n)}{m_{n+1}}- \frac{Z_n}{m_n}= \sum_{u\in\GG_n} F_n(u) ,
\]
where using $(\ref{ajj})$,
\[
F_n(u)=\frac{m_{n,n+1}(\ZZ(u))}{m_{n+1}}-\frac{1}{m_n}.
\]
\medskip

\noindent We use now standard arguments involving the genealogy of the population to exploit the branching property in the lineages. Given two individuals $u$ and $v$ living in generation $n$, let us consider their youngest common ancestor $\omega$ at generation $i-1$. Then, individuals $u$ and $v$ arise from two independent branching processes starting at time $i$ from descendants of $\omega$. We denote pairs of siblings born at generation $i$:
\[
\VV_{i,n} = \{(u,v) \in \GG_n : \exists \omega \in \GG_{i-1}, \ a \succcurlyeq 0, b \succcurlyeq 0 \ s.t. \ u \succcurlyeq \omega a, v\succcurlyeq \omega b, a_0 \neq b_0 \}.
\]

Then,
\begin{align*}
\E\left[\left(\sum_{u\in\GG_n} F_n(u) \right)^2\right] &= \E\left[\sum_{(u,v) \in \GG_n^2} F_n(u) F_n(v) \right] \\
&= \sum_{i=1}^{n} \E\left[\sum_{(u, v) \in \VV_{i,n}} F_n(u) F_n(v)\right] + \E\left[ \sum_{u\in\GG_n} F_n(u)^2 \right].
\end{align*}

On the one side,
\begin{align*}
\E\left[\sum_{u\in\GG_n} F_n(u)^2\right]
& \leq \frac{1}{m_n} + \frac{1}{m_{n+1}^2} \E\left[\sum_{u \in \GG_n} m_{n,n+1}(\ZZ(u))^2\right] \\
& \leq \frac{1}{m_n} + \frac{(m+1)^2 m_n}{m_{n+1}^2} \\
& \leq \left(1 + \frac{(m+1)^2}{m^2}\right) \frac{1}{m_n},
\end{align*}
recalling that $m_{n+1}\geq m_nm$ from Lemma \ref{lem:mean_behavior} $(iii)$.

On the other side, 
using 
\Bea
E_{i,n}(a)&=&\E\left(\sum_{u \in \GG_n} F_n(u) \ \vert \ \ZZ_{i}=\delta_{a}\right)\\
&=& \E\left(\frac{\sum_{u \in \GG_n^+} (m+1)+\sum_{u \in \GG_n^-} m}{m_{n+1}} -\frac{\# \GG_n}{m_n} \bigg\vert \ \ZZ_{i}=\delta_{a} \right)\\
&=&\frac{\E(\# \GG_{n+1} \vert \ \ZZ_{i}=\delta_{a} )}{m_{n+1}}-\frac{\E(\# \GG_{n} \vert \ \ZZ_{i}=\delta_{a})}{m_{n}}\\
&=&\frac{m_{i,n+1}(a)}{m_{n+1}}-\frac{m_{i,n}(a)}{m_n}
\Eea
and writing
$$V_{i,a,a',a''}:=\E\left( Z_{i}(a')Z_{i}(a'') \vert \ZZ_{i-1}=\delta_a\right) ,$$ 
we 
condition on the ages of the individuals at generation $i-1$ and $i$ to get
\Bea
&&\E\left[\sum_{(u, v) \in \VV_{i,n}} F_n(u) F_n(v)\right]\\
&&\quad =\E\left[\E\left[\sum_{(u, v) \in \VV_{i,n}} F_n(u) F_n(v)\ \bigg\vert \ \mathcal F_{i}\right] \right] \\
&& \quad = \sum_{a=0}^{a_{i-1}} m_{0,i-1}(0,a)\sum_{a', a'' \in \{0,a+1\}} V_{i,a,a',a''} E_{i,n}(a')E_{i,n}(a'').
\Eea
Finally, using Lemma~\ref{lem:ergodic} $(iv)$ and the fact that $V_{i,a,a',a''}$ is bounded by $\E((R+1)^2)$, 
\Bea
\left\| \frac{\E(Z_{n+1} \ \vert \ \mathcal F_n)}{m_{n+1}}-\frac{Z_n}{m_n} \right\|_2^2 
&\leq & \E((R+1)^2) C_m^2 \sum_{i=1}^{n} \frac{(1-c)^{2(n-i)}}{m_{i+1}} + C_R \frac{1}{m_n},
\Eea
where the constant $C_R$ depends only on $m$ and $\sigma$ and may change from line to line.
Then, splitting the sum in two parts and recalling that $(m_i)$ goes to infinity so is bounded from below,
\begin{equation}\label{eq:terme_3}
\left\| \frac{\E(Z_{n+1} \ \vert \ \mathcal F_n)}{m_{n+1}}-\frac{Z_n}{m_n} \right\|_2^2\leq C_R \left[\sum_{n/2\leq k \leq n }\frac{1}{m_k}+ (1-c)^n\right].
\end{equation}
Finally, putting~\eqref{eq:terme_1}, \eqref{eq:terme_2} and \eqref{eq:terme_3} in \eqref{eq:ineg} yields
\be \label{aveclesconstantes}
\left\| \frac{Z_{n+1}}{m_{n+1}} - \frac{Z_n}{m_n} \right\|_2\leq C_R \sum_{n/2\leq k \leq n }\frac{1}{\sqrt{m_k}}+ \left(1-c\right)^{n/2}.
\ee
Adding that $m_k$ grows geometrically from Lemma~\ref{lem:mean_behavior} $(iv)$, ensures that the right hand side is summable and \eqref{cvser} holds. We get the $L^2$ convergence of $Z_n/m_n$ and Cauchy Schwarz inequality also ensures that
\be
\label{ctrl1}
\sum_{n\in \N} \E(\vert Z_{n+1}/m_{n+1}-Z_n/m_n \vert)<\infty.
\ee
Then $Z_n/m_n$ converges a.s. and in $L^1$ to a non negative finite r.v. $W$. As $\E(Z_n/m_n)=1$, $\P(W>0)>0$. Adding that on Nonext, $Z_n$ and then $Z_n(0)$ have to go to infinity, a standard branching argument (see Lemma~\ref{lem:nonextinction} in Appendix for details) ensures that $\{W>0\}$=Nonext. This ends up the proof.
\end{proof}

\subsubsection{Proof of Theorem~\ref{thm:main} $(iii)$}\label{sec:main_iii}

To prove the third part of Theorem~\ref{thm:main}, we use the following  standard law of large numbers, in the vein of \cite{AK}, whose proof is provided for safe of completeness.

\begin{lemma}\label{lem:lgn}
Let $(X_n)_n$ be a sequence of integer valued r.v. We assume that for each integer $n$, $(X_{i,n} : i\geq0)$ are identically distributed centered r.v., which are independent of $X_n$. Moreover, we assume that $(X_{i,n} : i, n\geq 0)$ are bounded in $L^2$. Then,
\[
\frac{1}{X_n} \sum_{i=1}^{X_n} X_{i,n} \rightarrow 0 \quad \text{a.s on } \left\{ \liminf_{n\rightarrow \infty} \frac{1}{n} \log(X_n)>0 \right\}.
\]
\end{lemma}

\begin{proof}[Proof.]
Fix $c\in (0,\infty)$ and $\epsilon>0$ and denote by
$$E_n^{(c)}=\left\{ \left\vert \frac{1}{X_n} \sum_{i=1}^{X_n} X_{i,n} \right\vert \geq \epsilon \right\} \cap \{X_{n} \geq cn^2\}.$$
Using Bienaym\'e Tchebytchev inequality and $\E(X_{i,n} X_{j,n} \ \vert \ X_n)=\E(X_{i,n} \ \vert \ X_n)\E(X_{j,n} \ \vert \ X_n)=0$ when $i\ne j$, we get that
$$
\P(E_n^{(c)})\leq \E\left( \frac{1}{\eg^2 X_n} \ind_{X_n \geq cn^2} \sup_{i} \E(X_{i,n}^2)\right)\leq \frac{1}{c\eg^2n^2} \sup_{i} \E(X_{i,n}^2),
$$
where the supremum is bounded with respect to $n$ by assumption. As a consequence, $$\sum_{n\geq 0} \P(E_n^{(c)})<\infty$$
 and $\P(E_n^{(c)} \text{ occurs i.o.} )=0$ by Borel-Cantelli lemma. Then, a.s. on the event $\{\forall n\geq 0 : X_n\geq cn^2\}$,
$$\left\vert \frac{1}{X_n} \sum_{i=1}^{X_n} X_{i,n} \right\vert \stackrel{n\rightarrow \infty}{\longrightarrow} 0.$$
Adding that 
\[
\left\{ \liminf_{n\rightarrow \infty} \frac{1}{n} \log(X_n)>0 \right\} \subset \bigcup_{c\in \Q\cap (0,1)}\{\forall n\geq 0 : X_n\geq cn^2\}
\]
ends up the proof.
\end{proof}

\begin{proof}[Proof of Theorem~\ref{thm:main} $(iii)$]
First we prove that a.s on Nonext,
\be
\label{leszeros}
\frac{Z_{n+1}(0)}{Z_n} \stackrel{n\rightarrow \infty}{\longrightarrow} m
\ee
using the previous law of large numbers. Indeed, recall from equation~\eqref{eq:nv_nes_sans_im} that
$$Z_{n+1}(0) = \sum_{j=1}^{Z_{n}} R_{j,n+1},$$ 
where $(R_{j,n+1}, j \in \N)$ are i.i.d. and distributed as $R$, whose variance is finite. Moreover, Theorem \ref{thm:main} $(ii)$ and the fact that
$m_n \geq C\varrho^n$ for some $\varrho>1$ ensures that, on Nonext, $\liminf_{n\rightarrow \infty} \frac{1}{n} \log(Z_n)>0$. Then Lemma~\ref{lem:lgn} ensures that
$$\frac{Z_{n+1}(0)}{Z_n} \stackrel{n\rightarrow \infty}{\longrightarrow} \E(R)$$
a.s. on Nonext.
Now we just note that
$Z_n(a)=Z_{n-a}(0)$ to write
$$\frac{Z_n(a)}{Z_n}=\frac{Z_{n-a}(0)}{Z_{n-a-1}}\frac{Z_{n-a-1}}{Z_n}.$$
Using Theorem~\ref{thm:main} $(ii)$ and Lemma~\ref{lem:mean_behavior} $(ii)$ yields 
$$\lim_{n\rightarrow \infty} \frac{Z_{n-a-1}}{Z_n}=\lim_{n\rightarrow \infty} \frac{m_{n-(a+1)}}{m_n}=(m+1)^{-a-1}$$
a.s. on Nonext, which ends up the proof.
\end{proof}

\subsubsection{Proof of Corollary~\ref{cor:regimes}}\label{sec:cor_regimes}
First, recalling ~\eqref{eq:reproduction_sans_im}, we have
\begin{align*}
\E[Z_{n+1}|\FF_n] &= \E\left[\sum_{k=0}^{a_{n+1}} Z_{n+1}(k) | \FF_n\right] \\
&= \E\left[\sum_{k=0}^{a_{n+1}-1} Z_n(k) + \sum_{j=0}^{Z_n} R_{j,n+1} | \FF_n\right] \\
&= \sum_{k=0}^{a_{n+1}-1} Z_n(k) + m Z_n \\
&\leq (m + 1) Z_n,
\end{align*}
since $a_{n+1} - 1 \leq a_n$. Thus, $M_n = (m+1)^{-n} Z_n$ is a supermartingale.

Since $M_n$ is a non-negative supermartingale, it converges a.s. to a non-negative finite random variable $M$.\\
First, when $\liminf_{n\rightarrow \infty} a_n / \log  n>1/ \log(m+1)$, using Lemma~\ref{lem:mean_asymptotic} $(i)$, we get that $m_n$ is of the same magnitude as $(m+1)^n$. Part $(i)$ of Corollary~\ref{cor:regimes} is then a consequence of Theorem~\ref{thm:main}.\\
Similarly, Lemma~\ref{lem:mean_asymptotic} $(ii)$ ensures that $m_n/(m+1)^n$ goes to $0$ as $n\rightarrow \infty$ as soon as $ \limsup_{i\rightarrow \infty} a_n / \log n< 1/ \log(m+1)$ and we get $(ii)$.

\subsection{Aging branching process with immigration}\label{sec:immigration}

We now consider the aging branching process with immigration $(\bZ_n : n \in \mathbb N)$ defined in~\eqref{eq:reproduction} under  Assumptions~\eqref{as:age} and~\eqref{as:immigration}. Recall that $m_n$, defined in Equation~\eqref{def:min} and satisfying ~\eqref{eq:mn}, is the mean number of individuals in absence of immigration.

\begin{theorem}\label{thm:immigration}
$(i)$ The process $(\bZ_{n}/m_n : n \in \N)$ converges a.s. to a finite positive random variable. \\
$(ii)$ For  each  $a\geq 0$,  $(\bZ_n(a) / \bZ_n : n \in \N)$ converges a.s. to $m / (m+1)^{a+1}$.
\end{theorem}

 We remark that, contrary to the previous results, we are not needing the $L^2$ convergence of $\bZ_n / m_n$ in the sequel, whose proof seems quite technical and is not considered in this paper. Before the proof, we provide  the asymptotic behavior of the cumulative number of newborn individuals until generation $n$, which will be useful in next Section~\ref{sec:queueing}:
\[
\bbZ_n(0) = \sum_{k=0}^n \bZ_k(0).
\]
\begin{corollary}\label{thm:newborn}
The process $(\bbZ_n(0) / m_n : n \in \N)$ converges a.s. to a finite positive r.v.
\end{corollary}

\subsubsection{Proof of Theorem \ref{thm:immigration}}
We start with some technical additional results on the mean number $m_n$ of individuals in generation $n$.
Lemma~\ref{lem:ergodic} $(iv)$ ensures that for any $k\geq 0$,
$$\sum_{n\geq k} \sum_{a\leq a_n} \vert E(a,k,n) \vert <\infty$$
since $(a_n)$ grows at most linearly. Then
$$\sum_{n\geq k} \bigg\vert \frac{m_{k,n+1}}{m_{n+1}}- \frac{m_{k,n}}{m_{n}} \bigg\vert <\infty$$
and the following limit exists for $k\geq 0$,
\be
\label{cvdeux}
 \alpha_k=\lim_{n\rightarrow \infty} \frac{m_{k,n}}{m_n}.
\ee
We also remark that from Lemma~\ref{lem:mean_behavior} $(i)$, $m_{k,n}/m_n \leq 1/m_k(0,0)$. Thus, $\ag_k \leq 1/m_k(0,0)$. Adding that $m_k(0,0)\geq mm_{k-1}$, Lemma~\ref{lem:mean_behavior} $(iv)$ ensures that $(\ag_k)_k$ goes to $0$ geometrically and for any $\varrho \in (1,m+1)$, there exists $C>0$ such that
\be 
\label{vitesse}
\alpha_k\leq C\varrho^{k}.
\ee

\begin{proof}[Proof of Theorem \ref{thm:immigration} $(i)$]
 We use that $\bZ_n$ is the sum over $k\leq n$ of the descendants of immigrants arrived in generation $k$. We recall that notations $m_n$ and $m_{k,n}$ are introduced in Equation~\eqref{def:min}.

Moreover, the successive number of descendants of an immigrant $i$ in generation $k$ is equal to an aging branching process $(Z_{k,n}^{(i)}~:~n \geq k)$, without immigration, started at generation $k$, with maximal age $(a_n)_{n\geq k}$. Thus,
\be
\label{etoile}
\frac{\bZ_n}{m_n} = \sum_{k=0}^n \sum_{i=1}^{I_k}\frac{m_{k,n}}{m_n}W_{k,n}^{(i)} \quad a.s.,
\ee
where for each $k\leq n$,
\be
\label{defWi}
W_{k,n}^{(i)} := \frac{Z_{k,n}^{(i)}}{m_{k,n}}.
\ee
The sequence $(W_{k,n}^{(i)}, i \geq 1)$ is a sequence of i.i.d. r.v.\\
Theorem~\ref{thm:main} obtained in previous Section~\ref{sec:sans_immigration} ensures that for each $k \geq 0,i \geq 1$, 
\be
\label{cvun}
W_{k,n}^{(i)} \stackrel{n\rightarrow\infty}{\longrightarrow} W_k^{(i)} \qquad \text{a.s. and in } L^2.
\ee
We also note that $W_{k}^{(i)}$ are independent random variables for $k\geq 0,i\geq 1$ .

Using now the estimates~\eqref{aveclesconstantes} established in the proof of Theorem~\ref{thm:main}, there exists a constant $C_R$ (which depends neither of $k$, $i$ or $n$) such that
\[
\left\| W_{k+1,n}^{(i)} -W_{k,n}^{(i)} \right\|_2 \leq C_R \sum_{(n-k)/2\leq i \leq n-k }\frac{1}{\sqrt{m_{k,k+i}}}+ \left(1-c\right)^{(n-k)/2}.
\]
Since $(a_n)$ is non-decreasing, $m_{k,k+i} \geq m_i$ and $m_i$ grows geometrically. Thus,
\[
C_k := \sum_{n\geq k} \left\| W_{k+1,n}^{(i)} -W_{k,n}^{(i)}\right\|_2
\]
is bounded for $k\geq 0$. Then, by Cauchy Schwarz inequality,
$$\sup_{k} \E\left(\sum_{n\geq k} \vert W_{k+1,n}^{(i)} -W_{k,n}^{(i)} \vert\right)<\infty$$
and
\begin{equation}\label{eq:cte}
C := \sup_{k\geq 0} \E(\sup_{n\geq k} W_{k,n}^{(i)})<\infty.
\end{equation}

We can now prove that
\[
\frac{\bZ_n}{m_n} \stackrel{n\rightarrow\infty}{\longrightarrow} S := \sum_{k=0}^{+\infty} \sum_{i=1}^{I_k} \ag_k W_k^{(i)}.
\]

From the triangular inequality and $(\ref{etoile})$, for $1 \leq K \leq n$,
\begin{align*}
\left| \frac{\bZ_n}{m_n} - S\right| \leq &\sum_{k=0}^{K-1} \sum_{i=1}^{I_k} \left|\frac{m_{k,n}}{m_n} W_{k,n}^{(i)} - \ag_k W_k^{(i)}\right| \\
&+ \sum_{k=K}^n \sum_{i=1}^{I_k} \frac{m_{k,n}}{m_n} W_{k,n}^{(i)} + \sum_{k=K}^{+\infty} \sum_{i=1}^{I_k} \ag_k W_k^{(i)}.
\end{align*}
The first term goes a.s. \ to $0$ from the definitions of $\ag_k$ given in (\ref{cvdeux}) and $W_k^{(i)}$ in (\ref{defWi}). The last term is the rest of a convergent serie that goes to $0$ with $K$. This can be seen taking its expectation and combining \eqref{as:immigration} and \eqref{vitesse}. 
Then,
\[
\limsup_{n\to+\infty} \left|\frac{Z_n}{m_n} - S\right| \leq \mathcal{R},
\]
where $\mathcal{R}$ is the limit of the non-negative and non-increasing sequence
\begin{equation}
\mathcal{R}_K :=\sup_{n\geq K} \sum_{k=K}^n \sum_{i=1}^{I_k}\frac{m_{k,n}}{m_n}W_{k,n}^{(i)}.
\end{equation}

Thus, it remains to show that $\mathcal{R} = 0$. But,
\begin{align*}
\E[\mathcal{R}_K] &\leq \E\left[\sum_{k \geq K} \sum_{i=1}^{I_k} \frac{1}{m_k(0,0)} \sup_{n\geq k} W_{k,n}^{(i)}\right] \\
&\leq \sum_{k \geq K} \frac{\E[I_k]}{m_k(0,0)} \E[\sup_{n\geq k} W_{k,n}^{(i)}] \\
&\leq C \sum_{k \geq K} \frac{\E[I_k]}{m_k(0,0)},
\end{align*}
where constant $C$ is defined above in Equation~\eqref{eq:cte}.
Finally, since $m_k(0,0)$ goes to infinity faster than $\varrho^k$ for every $\varrho \in (1, m+1)$, Assumption~\eqref{as:immigration} on $I_k$ ensures that the rest of the serie goes to $0$ and $\mathcal R = 0$ a.s. Then $\bZ_n/m_n$ converges a.s. to a finite random variable $\bW$ which is non-negative (see Appendix~\ref{sec:nonextinction}).
\end{proof}

\begin{proof}[Proof of Theorem \ref{thm:immigration} $(ii)$]
 Following the  proof of  Theorem~\ref{thm:main} $(ii)$, we only prove that $\bZ_{n+1}(0) / \bZ_n$ goes to $m$ as $n$ goes to $\infty$.
First we recall from \eqref{eq:reproduction} that
\[
\frac{\bZ_{n+1}(0)}{\bZ_n} = \frac{1}{\bZ_n} \sum_{j = 1}^{\bZ_n} R_{j,n+1} + \frac{I_n}{\bZ_n}.
\]
From $(i)$, the process $(\bZ_n : n \in \N) $ satisfies the assumptions of the law of large numbers in Lemma~\ref{lem:lgn} and $\frac{1}{\bZ_n} \sum_{j = 1}^{\bZ_n} R_{j,n+1}$ goes to $m$ a.s. \\
It remains to show that $I_n / m_n$ goes to $0$. To do so, we use Borel-Cantelli Lemma. Indeed, for any $\eg > 0$,
\[
\P(I_n / m_n \geq \eg) \leq \frac{\E(I_n)}{\eg m_n}.
\]
Moreover Assumption~\eqref{as:immigration} ensures that there exists $\ell < m + 1$ such that for $n$ large enough, $\E[I_n] \leq \ell^n$. On the other side, from mean behavior Lemma~\ref{lem:mean_behavior}, for $\varrho \in (\ell, m+1)$ and $n$ large enough, there exists a positive constant $\alpha$ such that $m_n \geq \ag \varrho^n$. Thus, $\sum_n \P(I_n / m_n \geq \eg) < \infty$ and $I_n / m_n$ goes to $0$ a.s. This ends up the proof.
\end{proof}

\subsubsection{Proof of Corollary \ref{thm:newborn}}

The proof is in the same vein as the previous one and we just give the main lines.
We note that
\Bea
\frac{\bbZ_n(0)}{m_n} &=& \frac{\bZ_{0}(0)}{m_n} +\sum_{k=0}^{n-1} \frac{\bZ_{n-k}(0)}{m_n} \\
&=&\frac{\bZ_{0}(0)}{m_n}+ \sum_{k=0}^{n-1} \frac{\bZ_{n-k}(0)}{m_{n-k-1}}\frac{m_{n-k-1}}{m_n}.
\Eea
From previous Theorem~\ref{thm:immigration} $(ii$) and writing $S$ the limit of $\bZ_{n}/m_n$, for each $k\geq 0$,
$$\lim_{n\rightarrow \infty} \frac{\bZ_{n-k}(0)}{m_{n-k-1}} = \lim_{n\rightarrow \infty} \frac{\bZ_{n-k}(0)}{\bZ_{n-k-1}}\frac{\bZ_{n-k-1}}{m_{n-k-1}}=mS.$$
Recalling from Lemma~\ref{lem:mean_behavior} $(ii)$ that $m_{n-k-1}/m_n\rightarrow (m+1)^{-k-1}$ as $n\rightarrow \infty$, we obtain, for each $K > 0$,
$$\lim_{n\rightarrow \infty} \sum_{k=0}^K \frac{\bZ_{n-k}(0)}{m_n}=S m\sum_{k=0}^K (m+1)^{-k-1} = (1-1/(m+1)^{K+1}) S.$$
Then, letting $M:=\sup_{n\geq 0} \frac{\bZ_n}{m_n}$,
\Bea
\sum_{k=K}^{n-1} \frac{\bZ_{n-k}(0)}{m_n} &\leq& M \sum_{k=K}^n \frac{m_{n-k}}{m_n} \leq 
M \sum_{k=K}^n \frac{1}{m_k(0,0)}
\Eea
using Lemma \ref{lem:mean_behavior} $(i)$. Using $M<\infty$ a.s. from the previous Theorem and $m_k(0,0) = m m_{k-1}\geq m \alpha \varrho^{k}$ from Lemma \ref{lem:mean_branching} $(i)$ and Lemma \ref{lem:mean_behavior} $(iii)$,
 we get
$$\lim_{K \rightarrow \infty} \sup_{n\geq K} \sum_{k=K}^{n-1} \frac{\bZ_{n-k}(0)}{m_n} =0.$$
 This proves the a.s. convergence.

\section{From the motion of one single bus to aging branching process}\label{sec:bus_aging}
Links between polling systems and multitype branching processes with immigration have been successfully used for a finite number of queues, see e.g.~\cite{MeyResing,VdM, VD, Dyak, Vat}. We specify in this section the link between the bus line from Section~\ref{sec:introqueueing} and the aging branching process from Section~\ref{sec:introaging}. To simplify notations, we get rid of the dependance of the bus number and will denote respectively $M_n$ and $F_n$ the labels of the first and the last customers getting on the bus at the station $n$. Then, $H_n$ and $P_t$ are the time the bus leaves the station $n$ and its position at time $t$. \\

We write $R(t)$ the number of customers getting on the bus in a station when customers have been queuing during $t$ units of time before the bus arrival. The number of customers waiting at the station when the bus arrives 
is distributed as a Poisson distribution with parameter $\ag t$.
During the time these customers get on the bus, new customers may arrive and while these customers get on the bus, other customers arrive\ldots  We get a subcritical Galton-Watson Process whose reproduction law is a Poisson law with parameter $\ag$. Since $\ag < 1$, it will become extinct a.s. in finite time. Then, the number $R(t)$ of customers the bus takes at this station is a finite random variable, which is distributed as the total size of a subcritical branching process whose initial distribution 
 is a Poisson r.v. with parameter $\ag t$.
 First and second moments of this r.v. $R(t)$ can be found  in Appendix~\ref{sec:queues}. Here, a key role is played by the mean number $m$ of customers getting on the bus due to the time the bus has waited in a station while one customer was getting on:
$$m = \E[R(1)] = \frac{\ag}{1-\alpha}.$$

To construct the aging branching process with immigration associated with the bus queuing defined in \eqref{eq:ind}, we define the number of newborns
$\bZ_n(0)$ in generation $n$ as the number of customers getting on the bus 
at the station $n$:
\[
\bZ_n(0)=M_n-F_n+1 \qquad (n\geq 1).
\]
Letting $\bZ_0(0) = 0$, we observe that~\eqref{eq:ind} ensures that
\begin{equation}\label{eq:hn}
H_n = n \tau + \sum_{i=0}^n \bZ_i(0).
\end{equation}
\noindent Then, we introduce the maximal age 
\[
a_0=0, \quad a_n=d_{n+1}-2 \ \  (n\geq 1)  
\]
and we set, for any $a\leq a_n$,
\[
\bZ_n(a) = \bZ_{n-a}(0)\]
while $\bZ_n(a)=0$ if $a> a_n$. Observe also that $a_1=1$. Moreover,
\[
\bZ_{n} = \sum_{i=n-a_n}^{n} \bZ_i(0) = \sum_{i=0}^{a_n} \bZ_{n}(i),
\]

\noindent
The following result  shows that $\bZ_n$ is an aging branching process with immigration.
\begin{proposition}\label{bus_gw}
For every $n\geq 1$,
\[
\bZ_n(0) = \sum_{j = 1}^{\bZ_{n-1}} R_{j,n} + I_n,
\]
where $(R_{j,n},I_n\ : \ j, n \in \N)$ are independent r.v., $R_{j,n}$ are distributed as $R(1)$ and $I_n$ is distributed as $R(d_n \tau)$. 
\end{proposition}
\noindent Thus, $H_n - n \tau$ is the cumulated sum of the number of newborn until time $n$ for an aging branching process with maximal age $(a_{n})$, with reproduction law $R(1)$ and with 
immigration in generation $n$ distributed as $R(d_n \tau)$.

\begin{proof}[Proof of Proposition~\ref{bus_gw}]
The number customers who get on the bus at the station $n$ can be split as
$$\bZ_n(0)=N_n+I_n$$
where \\
\noindent $\bullet$ 
$N_n$ is the size of the population of customers   in station $n$  who arrived when  the bus was staying in the stations $i = n-d_n+1, \ldots, n-1$ and
their descendants, i.e. customers who arrived at the station 
$n$ while these customers are getting on the bus and so on... \\
\noindent $\bullet$  $I_n$ is the size of the population of  customers who arrived at the station $n$ during the time intervals when the bus was traveling between two stations $i$ and $i+1$, where $i=n-d_n,\ldots,n-1$, and their descendants.

Let us first characterize $N_n$.
The cumulative time the bus spent among the stations $i = n-d_n+1, \ldots, n-1$ is equal to the number of customers who got on the bus in one of these stations, so it is equal to
\[
\sum_{k=n-d_n+1}^{n-1} \bZ_k(0) = \bZ_{n-1}.
\]
We write $R_{j,n}$  the number of customers getting on the bus in station $n$  whose arrival is  associated to the $j$th unit time interval spent by the bus  among the stations  $i = n-d_n+1, \ldots, n-1$. Then the r.v.
$R_{j,n}$ are  distributed as $R(1)$ and  conditionally on $\bZ_{n-1}$,   $(R_{j,n} :  j=1, \ldots   ,\bZ_{n-1})$  is  independent of $(R_{j,k} : j\geq 0, k<n)$ since arrivals are  independent in each station.
Moreover,
\[
N_n= \sum_{j=1}^{\bZ_{n-1}} R_{j,n}, \qquad N_n\stackrel{d}{=}R(\bZ_{n-1})
\] 

We give now the distribution of  $I_n$. The cumulative time spent by the bus between the stations $i= n-d_n, \ldots, n$ is equal to $d_n\tau$.
Then the  number of customers $I_n$ who are  getting on the bus in the station $n$ and who are associated to  this cumulative time $d_n\tau$
is  distributed as  $R(d_n \tau)$. Finally, $I_n$  is independent of $(R_{k,n} : k,n \geq 0)$ using  the Poisson structure of arrivals in station $n$. 

Summing the sizes of these two subpopulations  gives the result.
\end{proof}

\section{Queuing system for a bus line}\label{sec:queueing}
We combine now the results of the two previous sections to study the time departures $(H_n)_n$ and motion
 $(P_t)_t$ of one single bus and  then of  two buses on the same line. Recall the connection $a_n=d_{n+1}-2$ between the maximal age and the queuing discipline.
\subsection{The motion of one single bus}\label{sec:buses_one}
We  prove the first statement given in Introduction using aging branching process with immigration. In the sequel, $d\geq 2$ is an integer  and we denote $\mathfrak{a}=d-2\geq 0$. 
\begin{theorem}\label{thm:1bus}
\noindent $(i)$ If $d_n = d \wedge n$ and $\varrho(d-2)<1$, then there exists $c>\tau$ such that
\[
H_n \sim_{n\rightarrow \infty} cn \quad a.s.
\]
Moreover, $(H_n-nc)/\sqrt{n}$ converges weakly to a Gaussian r.v. with positive variance, as $n\rightarrow \infty$.

\noindent $(ii)$ If $d_n = d \wedge n$ and $\varrho(d-2)>1$ or $d_n\rightarrow \infty$, then there exists a finite positive random variable $W$ \st
\[
H_n \sim_{n\rightarrow \infty} m_n W \quad a.s.
\]
\end{theorem}

\begin{proof}
For $(i)$, let $d_n=d\wedge n$ and $\varrho(d-2)<1$. The process defined for $n \in \N$ by $\mathscr{Z}_n = (\bZ_n(a) :  0 \leq a \leq \mathfrak{a})$ is a multitype
 branching process (with a finite number of types) and immigration $I_n$ in generation $n$ distributed as $R((d\wedge n) \tau)$. Let $M$ be the reproduction matrix associated to this
  process (see Appendix~\ref{sec:mean_asymptotics}) and $\bar{I} =(\E(R(d\tau)), 0,...,0)^*$. Since $\varrho(\mathfrak{a}) < 1$, Theorems 1.2 and 1.3 respectively in \cite{Roit} ensure that,
\[
\frac{\sum_{k=0}^{n} \mathscr{Z}_k}{n} \stackrel{n\rightarrow \infty}{\longrightarrow} b:=\sum_{k=0}^{\infty} M^k \bar{I} \qquad a.s.
\]
and $(\sum_{k=0}^n \mathscr{Z}_k - n b)/\sqrt{n}$ converges in distribution to a Gaussian r.v. Then we use this limit for the first coordinate of $\sum_{k=0}^n \mathscr{Z}_k$, which is $\sum_{k=0}^n \bZ_k(0)$ that goes to $b_1$. We get that $H_n/n$ goes a.s. to $\tau+b_1$ and the central limit theorem, thanks to its description~\eqref{eq:hn}. \\
The point $(ii)$ for $a_n\rightarrow \infty$ is a direct consequence of Corollary~\ref{thm:newborn}. Indeed, using the notation from this corollary for the aging branching process $\bZ$ and \eqref{eq:hn},
\[
H_n = n \tau + \bbZ_n.
\]
Moreover, from Proposition~\ref{bus_gw}, the immigration term $I_n$ is distributed as $R(d_n\tau)$, so $$\limsup_{n\rightarrow \infty} \frac{1}{n} \log(\E(I_n))=\limsup_{n\rightarrow \infty} \frac{1}{n} \log(d_n\tau m)=0<m+1,$$ since $d_n$ is bounded by $n$. Finally, since $d_n \geq 2$,
\[
\P(I_n \geq x) = \P(R(d_n\tau) \geq x) \geq \P(R(\tau) \geq x).
\]
We can now apply Corollary ~\ref{thm:newborn} to complete the proof. \\
The point $(ii)$ when $d_n=d\wedge n$ and $\varrho(d-2)>1$  can be handled similarly and is simpler since
 $\mathscr{Z}_n = (\bZ_n(a), 0 \leq a \leq \mathfrak{a})$ is a multitype branching process  (with a finite number of types). The counterpart Theorem \ref{thm:main}  (i-ii-iii) is actually 
 the classical Kesten Stigum theorem in the supercritical regime, see Theorem 1, Chapter V - Section 6 in \cite{AN}; while 
  the proofs of Theorem \ref{thm:immigration} and   Corollary~\ref{thm:newborn}  in that finite dimension cases  are the same. That  completes the proof.
\end{proof}

\begin{remark}
We stress that \cite{Roit} deals with multitype branching process in random environment with immigration and we just apply here  Theorems 1.2 and 1.3 in the constant environment case, where the assumptions can be easily verified since  we have $L^2$ moment assumptions. The results of next Corollary~\ref{corposit} could be extended to the case of $a_n$ random, which follows a stationary ergodic sequence. Let us note that in that case $b$ can have infinite components and new regimes appear, which could be of interest for a future work. 
\end{remark}

\begin{corollary} \label{corposit} $(i)$ If $d_n = d \wedge n$ and $\varrho(d-2)<1$, then $$
P_t\sim_{t\rightarrow \infty} \frac{t}{c} \qquad a.s.$$

\noindent $(ii)$ If $d_n=d\wedge n$ and $\varrho(d-2)>1$ or $d_n\rightarrow \infty$, then 
$$\limsup_{t\rightarrow \infty} \vert P_t - u_t \vert <\infty \quad a.s.,$$
where $u_t=\min( n\in \N : m_n\geq t)$ is the inverse function of $m_n$.
\end{corollary}

\begin{remark} \label{Rque} In the case $(ii)$, let us note that $m_n$ tends to infinite, so $u_t$ is well defined (finite) and goes to infinity. If $d_n=d\wedge n$ and $\varrho(d-2)>1$, then $u_t\sim \log(t)/\log(\varrho(d-2))$. If $\liminf_{n\rightarrow \infty} d_n / \log n > 1 / \log(m+1)$, then $m_n$ is of the same magnitude as $(m+1)^n$, so $u_t\sim \log(t)/\log(m+1)$ when $t\rightarrow \infty$.
\end{remark}

\begin{proof} The point $(i)$ is a direct consequence of Theorem \ref{thm:1bus} $(i)$.

For the point $(ii)$, we first use Theorem~\ref{thm:1bus} $(ii)$ to get that
$$H_n \sim_{n\rightarrow \infty} m_n W \qquad \text{a.s.}$$
with $W\in(0,\infty)$ a.s. \ We recall from Lemma~\ref{lem:mean_behavior} $(ii)$ that
$$\liminf_{n\to+\infty} \frac{m_{n+1}}{m_n} > 1.$$

To conclude, we prove now that for any sequences $v_n$ and $\w{v}_n$, if there exists $c>0$ such that 
$$v_n \sim c \w{v}_n \quad (n\rightarrow \infty), \qquad  \lim v_n = +\infty, \qquad \liminf \w{v}_{n+1} / \w{v}_n > 1,$$
 then $|\w{u}_t - u_t|$ is bounded where $u_t$ (resp. $\w{u}_t$) is the inverse function of $v_n$ (resp. $\w{v}_n$). Indeed, there exists $\ell > 1$, $c_1,c_2>0$ such that for $n$ large enough,
\[
\frac{\w{v}_{n+1}}{\w{v}_n} > \ell,  \quad \frac{v_{n+1}}{v_n} > \ell, \quad c_1 \w{v}_n \leq v_n \leq c_2 \w{v}_n.
\]
Then, for $t$ large enough, 
$v_{u_t-1}/\w{v}_{\w{u}_t}\geq  c_1 \w{v}_{u_t-1}/\w{v}_{\w{u}_t}$ and
 if additionally $u_t   \geq \w{u}_t+1$, 
\[
\frac{v_{u_t-1}}{\w{v}_{\w{u}_t}} \geq c_1 \ell^{u_t - 1 - \w{u}_t}.
\]
Furthermore, from the definition of the inverse functions,
$v_{u_t-1}/\w{v}_{\w{u}_t} \leq t/t = 1$, so
\[
1 \leq u_t - \w{u}_t \leq -\frac{\log c_1}{\log \ell} + 1.
\]
which end up the proof, since  the case
 $u_t - 1 \leq \w{u}_t$ is handled by symmetry.
\end{proof}

\subsection{Coalescence criteria for two buses with the same discipline}\label{sec:buses_two}

We can now determine when the two buses defined in the introduction coalesce a.s. The proof relies on a coupling argument, which enables  to compare these buses to two independent buses following the same discipline, whose behavior are described in the previous part.

\begin{theorem}\label{thm:coalescence}
$(i)$ If $d_n = d \wedge n$ and $\varrho(d-2)<1$, then  $\P(N<\infty)=1$.\\
$(ii)$ If $d_n = d \wedge n$ and $\varrho(d-2) > 1$, then $\P(N=\infty)>0$.\\
$(iii)$ If $d_n\rightarrow \infty$ as $n \rightarrow \infty$, then $\P(N<\infty)=1$.
\end{theorem}
First notice that the number of customers waiting for the bus $2$ at the successive stations is stochastically less than the number of customers for the bus $1$, since the customers obey the same discipline for the two buses, but some of them may have taken the bus $1$ instead of waiting for the bus $2$. Then, recalling that $\mu$ is the initial time delay between the two buses and  using the Poissonian structure of arrivals, we have for each $n\geq 0$,
\be
\label{coupling}
\{N>n\} \subset \{H_n^{(2)} \leq \tilde{H}_n^{(1)}+\mu \} \quad \text{a.s.},
\ee
where $(\tilde H_n^{(1)} : n \in \N)$ and $(H_n^{(1)} : n \in \N)$ are independent and identically distributed.
Roughly speaking, the proof works now as follows. In the first case $(i)$ the two buses move linearly, but the fluctuations of their positions make them coalesce. This can already be seen on the two independent buses. In the second case, the two buses may keep a gap larger than the discipline, which again can be seen using the positions of two independent buses. An event with an initial gap large enough is introduced for the coupling. Finally, in case $(iii)$, the second bus will be closer to the first bus than an independent bus would be: this ensures that the distance between them is bounded, so the second has much less customers to take than the first one; the two buses eventually coalesce. 

\begin{proof}[Proof of (i)]
 In this regime, we can use the central limit theorem of Theorem \ref{thm:1bus} $(i)$ for $(\tilde H_n^{(1)} : n \in \N)$ and $(H_n^{(1)} : n \in \N)$. So there exists $n_1 \in \N$ such that
\[
\P(\tilde{H}_{n_1}^{(1)} +\mu< n_1c) \geq 1/4 \text{ and } \P(H_{n_1}^{(1)}> n_1c) \geq 1/4.
\]
Thus, the coupling (\ref{coupling}) ensures that
$$\P(N \leq n_1) \geq \P(\tilde{H}_{n_1}^{(1)} +\mu<n_1c,\ H^{(1)}_{n_1}>n_1c) \geq 1/16.$$
We observe that $\mu$ can be taken random in the previous inequality, as soon as it is independent  of $({\bf T}{(i)} : i \geq 1)$. Moreover 
this result also holds if the discipline for the second bus is equal to $d$ for any station $n$, instead of $d\wedge n$, which means  that any customer arriving in the stations $i \in \{0,\ldots,  d\}$ before the arrival of bus $2$ stays in the queue. Indeed, this simply adds some customers for bus $2$ in the first $d$ stations, which yields an additional immigration term (only) in the first $d$ generations of the branching structure, with finite second moment.\\
We can then iterate the procedure and choose $n_2\in \N$ such that
$$\P( N \leq n_2 \vert N \geq n_1) \geq 1/16$$
and then inductively find $n_k\geq n_{k-1}$ for $k\geq 3$ such that
$$\P(N \geq n_k) \leq (1-1/16)^k.$$
Letting $k$ go to infinity ends up the proof. 
\end{proof}

\begin{proof}[Proof of $(ii)$] We show that
\be
\label{distreuiared}
\P( \exists t_0 >0, \forall t \geq t_0 : P_t^{(1)}-P_t^{(2)} > d) > 0,
\ee
so that the two buses can travel with the same discipline $d$ after time $t_0$, independently, as if they were on two different lines. 
We introduce the stopping time
$$
T_{d}:=\min(t \geq H^{(2)}_1, \ P_t^{(1)}-P_t^{(2)} \leq d), 
$$
with $\min \varnothing =\infty$ and we aim at proving that  $\P(T_{d} = \infty) > 0$.
First notice that, before time $T_{d}$, the bus $2$ is independent of the bus $1$ and with the same dynamics. The motion of the bus $2$ is indeed not (yet) accelerated by the fact that customers are taken by the bus $1$ instead of waiting
for the bus $2$. Thus,
$$P^{(2)}_{t- \mu}\stackrel{d}{=}P_t^{(1)}, \quad \text{for } t \in [\mu,T_{d}).$$
By Corollary \ref{corposit} $(ii)$, there exists an integer $\ell$ such that
\be
\label{probb}
 \P\left( \exists t \geq 0, \vert {P}^{(1)}_t-u_t \vert \geq \ell/2\right) \leq 1/2.
\ee
We introduce the bus $\hat{P}^{(1)}$ which takes the customers waiting at the stations $i \geq d+\ell$ with discipline $d_i=d\wedge i$, but those arrived at the station $d+\ell +j$ before time $\tau(d+\ell)$ for any $j\leq d$. It is defined recursively as follows.
For any $i\geq 1$,
\begin{align*}
\hat{H}_0^{(1)} &= 0\ a.s.\\
\hat{F}_i^{(1)} &= \min\left( j \geq 1 : T_j^{(d+\ell+ i)} >\hat{H}_{i-d_i}^{(1)}+\tau(d+\ell)\right) \\
\hat{\mathcal I}_i^{(1)} &= \{k \in \N : \forall j \leq k,\, T_{\hat{F}_i^{(1)} + j}^{(d+\ell+i)} \leq \hat{H}_{i-1}^{(1)} + \tau + j \} \\
\hat{M}_i^{(1)} &= \hat{F}_i^{(1)} -1 + \#\ \hat{\mathcal I}_i^{(1)} \\
\hat{H}_i^{(1)} &= \hat{H}_{i-1}^{(1)} + \tau + (\hat{M}_i^{(1)}-\hat{F}_i^{(1)}+1).
\end{align*}
By independence and stationarity of the Poisson Point Processes,
 $$\hat{P}^{(1)}\stackrel{d}{=} P^{(1)}$$
Let $N$ be a Poisson random variable of parameter $\ag \tau$ which is independent of the process of time arrivals $(\bT^{(i)} : i \geq 1)$ of customers. Without loss of generality, 
we assume that $\tau$ is an integer; otherwise one can replace $\tau$ by an integer larger than $\tau$. Denote by $E_\ell$ the event 
 $$E_{ \ell}= \bigcap_{1\leq i \leq d+\ell} 
 \{ T_1^{(i)} \geq i\tau \} \bigcap_{ d+\ell +1\leq i \leq 2d+\ell} 
 \{ T_1^{(i)} \geq (d+\ell)\tau \} \ 
 \bigcap \  \{ T_1^{(1)} > \mu, \  T_{(d+\ell)\tau+N}^{(1)} \leq \mu+\tau\}.
$$
which is the event when
\begin{itemize}
\item there are no customers at the stations $i=1,2,\ldots, d+ \ell$ when the bus $1$ arrives in these stations;
\item there are no arrivals of customers at the stations $i=d+\ell +1,\ldots,2d+\ell$ before the bus arrives at the station $d+\ell$;
\item the number of customers at the station $1$ arriving during the time interval $[\max(\mu,\tau),\mu+\tau]$ (so between the arrival of buses 1 and 2) 
 is at least $(d+\ell)\tau +N$.
 \end{itemize} 
 We first observe that $E_{\ell}$ has a positive probability and a.s. on $E_{\ell}$, we have
$$H_{d+\ell}^{(1)}=(d+\ell)\tau.$$
Adding that $\hat{P}^{(1)}$ and $E_{\ell}$ are constructed by restriction of the Poisson Point Process of arrivals to two disjoint domains, $\hat{P}^{(1)}$ and $E_{\ell}$ are independent.
Moreover, by construction,
 a.s. on the event $E_{ \ell}$, for any $t\geq 0$,
$$P^{(1)}_{t+(d+\ell)\tau}= \hat{P}^{(1)}_t+d+\ell.$$

\smallskip

Let $\hat{P}^{(2)}$ be a new bus, alone on the line. It starts from the station $0$ at time $\mu$, takes exactly $N$ customers in the station $1$ and then takes the customers waiting at the stations $i \geq 2$ which are given by the Poisson Point Process $(\bT^{(i)} -(d+\ell)\tau : i \geq 2) \cap [0,\infty)$ with discipline $d_i$ (associated to $a_i$). Then, 
$(\hat{P}^{(2)}_{\mu+t})_{t\geq 0}$ 
is distributed as $(P^{(1)}_{t})_{t\geq 0}$. Finally, since the bus $\hat{P}^{(2)}$ has the same customers than $P^{(2)}$ to take at the stations 
$i\geq 2$ before time $T_{d}$ and $(d+\ell)\tau$ customers less to take in the station $1$ on the event $E_{\ell}$, we have 
$$P^{(2)}_{t+(d+\ell)\tau}= \hat{P}^{(2)}_{t} \qquad \text{a.s. on } E_{ \ell}$$
for any $t\in [H^{(2)}_1, T_{d}]$.
\smallskip
Thus
\Bea
&&\{E_{ \ell}, \ \forall t \geq 0, \ \hat{P}^{(1)}_t- \hat{P}^{(2)}_{t} > -\ell\}\\
&& \qquad
=\{E_{ \ell}, \ \forall t \geq (d+\ell)\tau, \ (\hat{P}^{(1)}_{t-(d+\ell)\tau}+d+\ell) - \hat{P}^{(2)}_{t-(d+\ell)\tau} > d\}\\
&& \qquad \quad \subset \{E_{ \ell}, \ \forall t \geq (d+\ell)\tau, \ P^{(1)}_t- P^{(2)}_t > d, T_{d} =\infty\} \subset \{T_{d} =\infty\}.
\Eea
Finally,
\Bea
\P(T_{d}=\infty)& \geq& \P( E_{ \ell}) \P\left( \forall t \geq 0, \ \hat{P}^{(1)}_t-\hat{P}^{(2)}_t >-\ell\right),
\Eea
and 
\Bea
&&\P\left( \forall t \geq 0, \ \hat{P}^{(1)}_t-\hat{P}^{(2)}_t >-\ell\right) \\
&& \qquad \geq 1-\P(\exists t \geq 0 :  \hat{P}^{(1)}_t -u_t \leq -\ell/2 \ \emph{or} \  u_t- \hat{P}^{(2)}_t \leq -\ell/2)\\
&& \qquad \geq 1-\P(\exists t \geq 0 : \vert {P}^{(1)}_t -u_t\vert \geq \ell/2)>0
\Eea
using $(\ref{probb})$. We obtain that $\P(T_{d}=\infty)>0$, which ends up the proof. 
\end{proof}

\begin{proof}[Proof of (iii)]
 We use again the coupling (\ref{coupling}) and Corollary \ref{corposit} $(ii)$ to get
\[
\limsup_{t\rightarrow \infty} \ P_t^{(1)}-P_t^{(2)} \leq \limsup_{t\rightarrow \infty} \ P_t^{(1)}-\tilde {P}_t^{(1)} <\infty \qquad a.s.
\]
Then the non-decreasing sequence of events
$$A_K:=\{\forall t>0 : P_t^{(1)}-P_t^{(2)} \leq K\}$$
satisfies $\P(\cup_{K\in \N} A_K)=1$. Moreover, a coupling for the second bus ensures that
$$\P( A_K, \forall t>0 : P_t^{(1)}>P_t^{(2)})\leq \P( \forall t>0 : P_t^{(1)}>\hat{P}_t^{(2)}),$$
where $\hat{P}^{(2)}$ is independent of $P^{(1)}$ and gives the position of a single bus on the line with discipline $\widehat{d_n}=K\wedge n$. Using again Corollary
 \ref{corposit} $(ii)$ with this discipline $\widehat{d_n}=K\wedge n$, we obtain a logarithm speed $u_t^{(2)}\sim \log(t) / \log(\varrho(K-2))$ for $\hat P^{(2)}$ as $t\rightarrow \infty$, 
 see Remark \ref{Rque}. Adding that
 $P^{(1)}$ goes slowler and more precisely  that $u_t^{(1)}\sim \log(t) / \log(\varrho)$ as $t\rightarrow \infty$ for some $\varrho>\varrho(K-2)$, we get
 $$\lim_{t\rightarrow \infty} \hat{P}_t^{(2)}-P^{(1)}_t=\infty \qquad \text{a.s.}$$
 So $\P( A_K, \forall t>0 : P_t^{(1)}>P_t^{(2)})=0$ and letting $K\rightarrow \infty$,
 $$\P( \forall t>0 : P_t^{(1)}>P_t^{(2)})=0,$$
which ends up the proof. 
\end{proof}

\appendix
\section{Mean behaviors of aging branching process}\label{sec:mean_asymptotics}
Let $\varrho(a)$ be the greatest eigenvalue of the squared Leslie's matrix of size $a+1$ governing a population with maximal age $a$
 $$\begin{pmatrix}
 m & m & \cdots & m \\
 1 & 0 & \cdots & 0 \\
 \vdots & \ddots & \ddots & \vdots \\
 0 & \cdots & 1 & 0
 \end{pmatrix},$$
 i.e. the greatest root of $P_a = X^{a+1} - m \sum_{k = 0}^a X^k$. Thus,
\begin{equation}\label{eq:rhoa}
\varrho(a)^{a+1} = m \sum_{k=0}^a \varrho(a)^k.
\end{equation}
Obviously,  $(\varrho(a))_{a\geq 1}$ is non-decreasing. Let us  specify its asymptotic behavior. When $a$ goes to infinity, individuals become immortal
 and we recover a Galton-Watson process with mean  $m+1$. We give here the second order term.
\begin{lemma}\label{lem:technical}
When $a$ goes to infinity,
\[
\varrho(a) = (m + 1) - m (m + 1)^{-a-1} + O(a (m+1)^{-2a}).
\]
\end{lemma}

\begin{proof}
Notice that $P_a(X) = Q_a(X)/(X - 1)$ where $Q_a(X)=X^{a+1}(X-m-1) + m$. Adding that $Q_a'((a+1)(m+1)/(a+2))=0$ and $Q_a(m+1)>0$ yields
\[
\frac{a+1}{a+2} (m + 1) \leq \varrho(a) \leq m+1.
\]
Thus, $\varrho(a)$ goes to $m + 1$. 

\medskip

Writting $\varrho(a) = (m + 1) (1 - \eg_a)$ and raising the previous inequality at exponent $a+1$ yields 
\[
\left(1 - \frac{1}{a+2}\right)^{a+1} \leq \left(1 - \eg_a\right)^{a+1} \leq 1
\]
and $\big(\left(1 - \eg_a\right)^{a+1} : a\geq 1\big)$ is bounded.
Since $Q_a(\varrho(a)) = 0$, then $(m + 1)^{a+2} (1 - \eg_a)^{a+1} \eg_a = m$. 
First,
\[
(a+1) \eg_a = m \cdot (a+1) \cdot (m+1)^{-a-2} \cdot (1-\eg_a)^{-a-1} \to 0.
\]
Then,
\[
\eg_a (m + 1)^{a + 2} = m e^{-(a + 1) \ln(1 - \eg_a)} \to m.
\]
Going one step further,
\begin{align*}
\eg_a (m + 1)^{a + 2} - m &= m e^{(a+1) \eg_a + o((a+1) \eg_a)} \\
&= m (a + 1) \eg_a + o((a + 1) \eg_a) \  \sim \ m^2 (a + 1) (m + 1)^{-a-2}.
\end{align*}
Finally,
\[
(m + 1) - \varrho(a) - m (m + 1)^{-a-1} \sim m^2 (a + 1) (m + 1)^{-2a-3}
\]
as $a\rightarrow \infty$.
\end{proof}

\begin{lemma}\label{lem:prod_leslie}
$(i)$ If $\sum (m+1)^{-a_i} = +\infty$, then
\[
\prod_{i=1}^n \varrho(a_i) = o((m+1)^n)\quad (n\rightarrow \infty).
\]
$(ii)$ If $\sum (m+1)^{-a_i} < +\infty$, then there exists $\ag > 0$ \st
\[
\prod_{i=1}^n \varrho(a_i) \sim \ag  (m + 1)^n \quad (n\rightarrow \infty).
\]
\end{lemma}

\begin{proof}
Let us write
\[
(m+1)^{-n} \prod_{i=1}^n \varrho(a_i) = \exp\left\{ \sum_{i=1}^n \log \frac{\varrho(a_i)}{m+1} \right\}.
\]
Using the expansion from Lemma~\ref{lem:technical}, recall that $(m+1)^{-1} \varrho(a_i) = 1 - m \cdot (m+1)^{-a_i-2} + O(a_i (m+1)^{-2 a_i})$ and its logarithm behaves like $(m+1)^{-a_i}$.
\end{proof}

\begin{lemma}
\label{lem:mean_asymptotic} 
$(i)$ If $\sum (m+1)^{-a_i} < +\infty$, then there exists $\beta > 0$ \st
\[
\beta (m + 1)^n \leq m_n \leq (m+1)^n.
\]
$(ii)$ If $\limsup_{i\rightarrow \infty} a_i / \log i < 1 / \log(m + 1)$, then
\[
m_n = o((m+1)^n).
\]
\end{lemma}

\begin{remark}
The first condition is valid as soon as $\liminf a_i / \log i > 1 / \log(m + 1)$.
\end{remark}

\begin{proof}[Proof of $(i)$ - Upper bound]
Branching process $(Z_n : n\in\N)$ is stochastically bounded by a Galton-Watson process with reproduction law $R+1$. Thus, for every integer $n$, $m_n \leq (1 + m)^n$.
\end{proof}

\begin{proof}[Proof of $(i)$ - Lower bound.]
Since $m_{n+1}(0,0) = m m_n$, we just need to consider $m_n(0, 0)$.
We prove by induction that $\prod_{j=0}^{n-1} \varrho(a_j) \leq m_n(0, 0)$ and conclude using   Lemma~\ref{lem:prod_leslie}. Recall that $m_0(0) = 1 = \varrho(0)$. Then, for $k \in \entiers{n - a_n}{n}$, by induction
\begin{align*}
m_k(0,0) &\geq \prod_{j=0}^{k-1} \varrho(a_j) \\
&\geq \prod_{j=0}^{n-1} \varrho(a_j) \left[\prod_{j=k}^{n-1} \varrho(a_j)\right]^{-1}
\geq \prod_{j=0}^{n-1} \varrho(a_j) \cdot \varrho(a_n)^{k-n},
\end{align*}
since sequence $(a_j)$ and function $\varrho$ are non decreasing.  We obtain
\[
m_{n+1}(0,0)=m \sum_{k=n-a_n}^n m_k(0) \geq \prod_{j=0}^{n-1} \varrho(a_j) \left[ m \sum_{k=n-a_n}^n \varrho(a_n)^{k-n} \right]
= \prod_{j=0}^n \varrho(a_j).
\]
which ends up the induction.
\end{proof}

\begin{proof}[Proof of $(ii)$]
Let us show, by induction, that
\[
\forall n \in \N,\, m_n(0, 0) \leq \varrho(a_n)^n.
\]
First, $m_0(0) = \varrho(0) = 1$.
Then, since $\varrho$ and $(a_n)_n$ are increasing, for every $k \in \entiers{n-a_n}{n}$,
\begin{align*}
m_k(0)  \ \leq \ \varrho(a_k)^k  \ \leq \ \varrho(a_n)^k.
\end{align*}
Thus,
$
m \sum_{k = n-a_n}^n m_k(0) \leq m \sum_{k=n-a_n}^n \varrho(a_n)^k $
and
$$m_{n+1}(0) \leq \varrho(a_n)^{n+1} \leq \varrho(a_{n+1})^{n+1}.$$
Finally, from Lemma~\ref{lem:technical}, $(m + 1)^{-n} \varrho(a_n)^n = \exp\{n \log(1 - m (m+1)^{-a_n-2} + o(a_n(m+1)^{-2a_n})\}$, which gives the desired result.
\end{proof}

\section{Dichotomy for branching processes}\label{sec:nonextinction}

Let $(Z_n : n \in \N)$ be an aging branching process (without immigration) with maximal age $(a_n)$ as defined in Section~\ref{sec:aging} and introduce the stopping times
$$T_K=\inf\{n \in \N : Z_n(0) \geq K\}$$
where $\inf \varnothing =+\infty$ by convention and
$$S = \bigcap_{K \geq 1} \{T_K <\infty\}=\left\{ \limsup_{n\rightarrow \infty} Z_n(0) =+\infty\right\}.$$
We recall that Nonext$=\{ \forall n \in \N : Z_n\ne 0\}$ and we know from Section \ref{sec:sans_immigration} that
$Z_n/m_n$ converges to a finite random variable $W$ which is positive with positive probability.

\begin{lemma}\label{lem:nonextinction}
(i) $S = \emph{Nonext}$. \\
(ii) A.s. on the event \emph{Nonext}, we have
$$\lim_{n \rightarrow \infty} \frac{Z_n}{m_n}>0.$$
\end{lemma}
\begin{proof}
$(i)$ is a classical result of dichotomy for branching processes. 
Either $\P(R=0)=0$ and $0$ is not accessible for $Z_n(.)$ and $\P(\text{Nonext})=1$. Then we use  the same coupling as in Section \ref{sec:sans_immigration}, i.e. a multitype branching process with fixed maximal age $a$, such that $\varrho(a)>1$. 
Then, the branching process $Z_n$ is larger than a supercritical multitype branching process with finite variances and primitive mean matrix and no extinction. It ensures that each of its coordinates goes to 
 $\infty$  a.s. \\
Or  $\P(R=0)=0$ and $0$ is  accessible for $Z_n(.)$. Then
  $$\limsup_{n\rightarrow \infty} Z_n=\infty \quad \text{a.s. on Nonext}.$$
 which ensures that  $\limsup_{n\rightarrow \infty} Z_n(0)=\infty \quad \text{a.s. on Nonext}.$
Finally, the converse inclusion is obvious.\\
$(ii)$ By branching property and monotonicity of the sequence $(a_n)$,
$$Z_{T_K+n}\geq \sum_{i=1}^{K} Z_n^{(i)},$$
a.s. on $\{T_K <\infty\}$, where $(Z^{(i)} : i \geq 1)$ are i.i.d. aging branching processeses whith maximal age is given by $(a_n)$. Adding that
$$\P\left(\lim_{n\rightarrow \infty} \frac{Z_n}{m_n} >0\right)=\P(W>0)>0$$
we get
$$\P\left( \liminf_{n\rightarrow \infty} \frac{Z_{T_K+n}}{m_n} >0 \ \vert \ T_K <\infty\right) \geq 1-(1-\P(W>0))^K$$
Finally $T_K<\infty$ and $m_{n}/m_{n+T_K}$ has a positive lower bound as $n\rightarrow \infty$ thanks to Lemma 2.2 $(ii)$, so
$$\P\left( \liminf_{n\rightarrow \infty} \frac{Z_{n}}{m_n} >0 \ \vert \ T_K <\infty \right)=\P\left( \liminf_{n\rightarrow \infty} \frac{Z_{T_K+n}}{m_n} >0 \ \vert \ T_K <\infty\right)$$
and we get 
$$\P\left( \liminf_{n\rightarrow \infty} \frac{Z_{n}}{m_n} >0 \ \vert \ T_K <\infty\right)\geq 1-(1-\P(W>0))^K.$$
Adding that the right hand goes to $0$ as $K$ goes to infinity ensures that
$$\P\left(\liminf_{n\rightarrow \infty} \frac{Z_{n}}{m_n} =0; \ S\right)=0$$ and ends up the proof.
\end{proof}

We have the following counterpart for aging branching process with immigration $(\bZ_n)_{n\in \N}$. The proof is similar and omitted.

\begin{lemma} 
We have $\P(\emph{Nonext})=1$ and 
$$\lim_{n \rightarrow \infty} \frac{\bZ_n}{m_n}>0 \qquad \text{a.s}.$$
\end{lemma}

\section{M/G/1 queues}\label{sec:queues}

To be self-contained, we recall below some well known facts around queues.
We want to determine the time $R$ a bus wait at a station where one customer was waiting when it arrived. We recall that, more generally, $R(t)$ is the time a bus wait at a station where customers gathered from $t$ units of times when the bus arrived. We use the classical connection with Galton-Watson processes.  For each customer getting on, his children are the customers arriving during his boarding. Since the customers arrivals are led by a Poisson Point process, the number of children is distributed as i.i.d. r.v. with Poisson distribution with parameter $\ag$. Intuitively, $R < +\infty$ a.s. if and only if customers arrive at the station slower than they board.
\begin{proposition}\label{prop:esperance_attente}
$R < \infty$ a.s. iff $\ag \leq 1$. Then,
\[
\E[R(t)]=\frac{\ag}{1-\ag} t.
\]
Furthermore,
\[
E\left[R(t)^2\right] = \frac{\ag}{(1-\ag)^3} t + \frac{\ag^2}{(1-\ag)^2} t^2.
\]
\end{proposition}

\begin{proof}
To enlight notations, let us still denote $(Z_n~:~n\in\N)$ the standard Galton-Watson process described just above. Then,
$
R = \sum_{i=0}^\infty Z_i.$
Thus, $R < +\infty$ a.s. if and only if $\ag \leq 1$. Furthermore,
\[
\E[R] = \sum_{i=0}^\infty \E(Z_i) = \sum_{i=0}^\infty \ag^i = \frac{1}{1-\ag}.
\]

The number of customers gathering during a time interval of length $t$ when the bus arrives is distributed as $N(t)$, a r.v. with Poisson distribution with parameter $\ag t$. Then, 
\[
\E[R(t)] = \E[N(t)] \E[R] = \frac{\ag t}{1-\ag}.
\]

In the same way, using independence and denoting $(R^{(i)}, i \in \N)$ a sequence of i.i.d. r.v. with common law $R$, we can compute the second moment:
\[
\E[R(t)^2] = \E \left[ \left(\sum_{i=1}^{N(t)} R^{(i)} \right)^2 \right]
 = \E[N(t)] \E[R^2] + \E[N(t)(N(t)-1)] \E[R]^2.
\]
Finally, using that $R$ and $1 + R(1)$ are equal in distribution,
$
\E[R^2] = 1/(1-\ag)^3,$
and the proof is complete.
\end{proof}

$\newline$
\textbf{Acknowledgement.} This work was partially funded by Chair
Modelisation Mathematique et Biodiversite VEOLIA-Ecole
Polytechnique-MNHN-F.X. and the professorial chair Jean Marjoulet.

\bibliography{bus_biblio}

\end{document}